\documentclass[11pt,letterpaper]{amsart}

\usepackage[]{amsmath, amsthm, amsfonts, verbatim, amssymb}
\usepackage[letterpaper, left=1in, right=1in, top=1in, bottom=1in]{geometry}

\usepackage[backref=page]{hyperref} 
\hypersetup{colorlinks = true} 
\usepackage{rotating}
\usepackage{tabularx} \usepackage{booktabs}
\usepackage{multirow}
\usepackage{epstopdf}
\usepackage{pifont}
\usepackage{enumerate}
\usepackage{soul}
\usepackage{graphicx}
\usepackage{mathrsfs}\usepackage[all]{xy}
\DeclareMathAlphabet{\mathpzc}{OT1}{pzc}{m}{it}
\usepackage{tikz,tikz-cd,color}
\usepackage[makeroom]{cancel}
\usetikzlibrary{arrows}

\renewcommand*{\backref}[1]{}
\renewcommand*{\backrefalt}[4]{%
	\ifcase #1 (Not cited.)%
	\or        (Cited on page~#2.)%
	\else      (Cited on pages~#2.)%
	\fi}

\newtheorem{theorem}{Theorem}[section]
\newtheorem*{theorem*}{Theorem}
\newtheorem{lemma}[theorem]{Lemma}
\newtheorem{proposition}[theorem]{Proposition}
\newtheorem{corollary}[theorem]{Corollary}
\newtheorem{conjecture}[theorem]{Conjecture}
\newtheorem{definition}[theorem]{Definition}

\theoremstyle{remark}
\newtheorem{remark}[theorem]{Remark}
\newtheorem{example}[theorem]{Example}

\newcommand{\cc}{\mathcal{C}}
\newcommand{\pp}{\mathbb{P}}

\newcommand{\D}{\mathcal{D}}

\newcommand{\oo}{\mathcal{O}}

\newcommand{\ff}{\mathbb{F}}

\author[A.~Aravena]{Anibal Aravena}
\address{University of Massachusetts Amherst, Department of Mathematics \& Statistics, Amherst, MA 01003, USA
}
\email{aaravena@umass.edu}

\author[J.~Negrete]{Jaime Negrete}
\address{University of Georgia, Department of Mathematics, Athens, GA 30602, USA
}
\email{jaime.negrete@uga.edu}

\author[W. Yeong]{Wern Yeong}
\address{University of California, Los Angeles, Department of Mathematics, Los Angeles, CA 90095, USA
}
\email{wyyeong@math.ucla.edu}

\begin{document}
\title[]{Pseudo-hyperbolicity of Horikawa surfaces}


\begin{abstract}
Horikawa surfaces are minimal complex algebraic surfaces of general type with minimal Chern slope, satisfying either $c_2=5c^2_1+36$ if $c_1^2$ is even, or $c_2=5c^2_1+30$ if $c_1^2$ is odd.
We prove that very general Horikawa surfaces with $p_g\ge 5$ contain only finitely many rational or elliptic curves. 
Moreover, we provide an explicit characterization and count of these curves.
Our results also apply to very general Horikawa surfaces of the first kind with $p_g\in \{3,4\}.$
\end{abstract}

\maketitle

\vspace{-10pt}

\setcounter{tocdepth}{2}

\tableofcontents

\vspace{-30pt}

\section{Introduction}

A smooth complex projective variety $X$ is said to be \emph{Brody hyperbolic} if it does not contain any non-constant entire holomorphic curves.
Lang conjectured that $X$ is Brody hyperbolic if and only if every subvariety of $X$, including $X$ itself, has a desingularization of general type \cite{Lan86}. 
This would be a consequence of the following conjecture.

\begin{conjecture}[Green--Griffiths--Lang \cite{GrGr80, Lan86}]
\label{conj:ggl}
{\em
A smooth complex projective variety $X$ of general type is \emph{pseudo-Brody hyperbolic}, i.e. there exists a proper algebraic subset of $X$ containing all non-constant entire curves in $X$.
}
\end{conjecture}

These two conjectures are still open in full generality, and despite major strides, they remain open even in the case of surfaces.
Conjecture~\ref{conj:ggl} is known to hold for surfaces satisfying $c_1^2>c_2$ \cite{Bog78, LuYa90, McQ98} and for surfaces of irregularity at least two  \cite{Kaw81,NWY02, Lu10}.
Note that smooth surfaces of general type in $\mathbb{P}^3$ (degree $d\ge5$) do not satisfy $c_1^2>c_2$ and have irregularity zero.
Nonetheless, it is known that a very general such surface does not contain rational or elliptic curves, and is in fact Demailly algebraically hyperbolic (see Definition~\ref{def:alg-hyp}) \cite{Cle86, Ein88, Xu94, CoRi19}.
Although Demailly algebraic hyperbolicity is a necessary condition for Brody hyperbolicity, it remains unknown whether very general surfaces in $\mathbb{P}^3$ of low degree, such as quintics, are Brody hyperbolic.
For higher degrees, Demailly and El Goul \cite{DeEl00} proved that a very general surface in $\pp^3$ of degree $d\ge21$ is indeed Brody hyperbolic.

Concerning mildly singular surfaces, Bogomolov and De Oliveira \cite{BoDe06} proved that a nodal surface in $\pp^3$ with a sufficiently large number of nodes relative to its degree contains only finitely many rational or elliptic curves.
Roulleau and Rousseau \cite[Theorem 1]{RoRo14} proved that this finiteness also holds for certain canonical surfaces satisfying $c_1^2>\frac{3}{5}c_2$. 
Furthermore, there is a rich literature on constructing explicit examples of surfaces that are (pseudo-)hyperbolic; see, for example, \cite{GaUr20} for complete intersection surfaces in projective space, and the references therein.

In view of Conjecture~\ref{conj:ggl} and the known cases discussed above, Horikawa surfaces are of particular interest because they are general type surfaces with minimal $c_1^2.$

\begin{definition}
{\em
A \emph{Horikawa surface} is a minimal surface of general type whose Chern numbers satisfy either:
\begin{itemize}
\item $c_2 = 5c_1^2+36$ (equivalently, $c_1^2=2p_g-4$), if $c_1^2$ is even; or
\item $c_2 = 5c_1^2+30$ (equivalently, $c_1^2=2p_g-3$), if $c_1^2$ is odd.
\end{itemize}
A Horikawa surface is said to be of the \emph{first kind} or \emph{second kind} according to whether $c_1^2$ is even or odd, respectively.
}
\end{definition}

In the geography of Chern numbers for minimal surfaces of general type (see Figure~\ref{fig:chern-geography}), Horikawa surfaces lie on or near the boundary given by the Noether inequality $5c_1^2-c_2+36\geq0$ (equivalently, $c_1^2\ge 2p_g-4$).
Since these surfaces have minimal Chern slope $c_1^2/c_2\approx \frac{1}{5}$, they are the furthest of any general type surface from satisfying $c_1^2>c_2$.
Furthermore, they still fall outside the aforementioned relaxed bound of $c_1^2>\frac{3}{5}c_2$ \cite[Remark 18]{RoRo14}.
Consequently, the (pseudo-)hyperbolicity of Horikawa surfaces is expected to be among the most difficult to establish in dimension two.

\begin{figure}[htpb]
\centering
\begin{tikzpicture}[scale = 0.28, x=.1cm, y=.1cm, >=stealth]

\fill[black!8] (0,0) -- (100,300) -- (300,300) -- (300,52.8) -- (36,0) -- cycle;

\fill[green!15] (0,0) -- (100,300) -- (300,300) -- cycle;

\draw[->, thick] (-10,0) -- (320,0) node[right] {$c_2$};
\draw[->, thick] (0,-10) -- (0,320) node[above] {$c_1^2$};

\foreach \x in {50, 100, 150, 200, 250, 300} {
    \draw (\x, 4) -- (\x, -4) node[below=2pt] {\x};
}
\foreach \y in {50, 100, 150, 200, 250, 300} {
    \draw (4, \y) -- (-4, \y) node[left=2pt] {\y};
}

\draw[very thick, black] (0,0) -- (100,300) 
    node[pos=0.85, sloped, above] {$c_1^2 \le  3c_2$ (BMY)};

\draw[very thick, dashed, blue] (0,0) -- (150,300) 
    node[pos=0.8, sloped, above] {$c_1^2 > 2c_2$ (Lu--Yau \cite{LuYa90})};

\draw[very thick, dashed, orange!90!black] (0,0) -- (300,300) 
    node[pos=0.7, sloped, above] {$c_1^2 > c_2$ (Bogomolov \cite{Bog78}, McQuillan \cite{McQ98})};

\draw[very thick, red!80!black] (36,0) -- (300,52.8) 
    node[pos=0.75, sloped, above] {$5c_1^2 \ge c_2 - 36$ (Noether)};

\end{tikzpicture}
\caption{The geography of Chern numbers for minimal surfaces of general type. 
Valid Chern numbers are bounded by the Bogomolov--Miyaoka--Yau (BMY) and Noether inequalities. 
Conjecture~\ref{conj:ggl} is known to hold for surfaces in the green-shaded region ($c_1^2 > c_2$).}
\label{fig:chern-geography}
\end{figure}
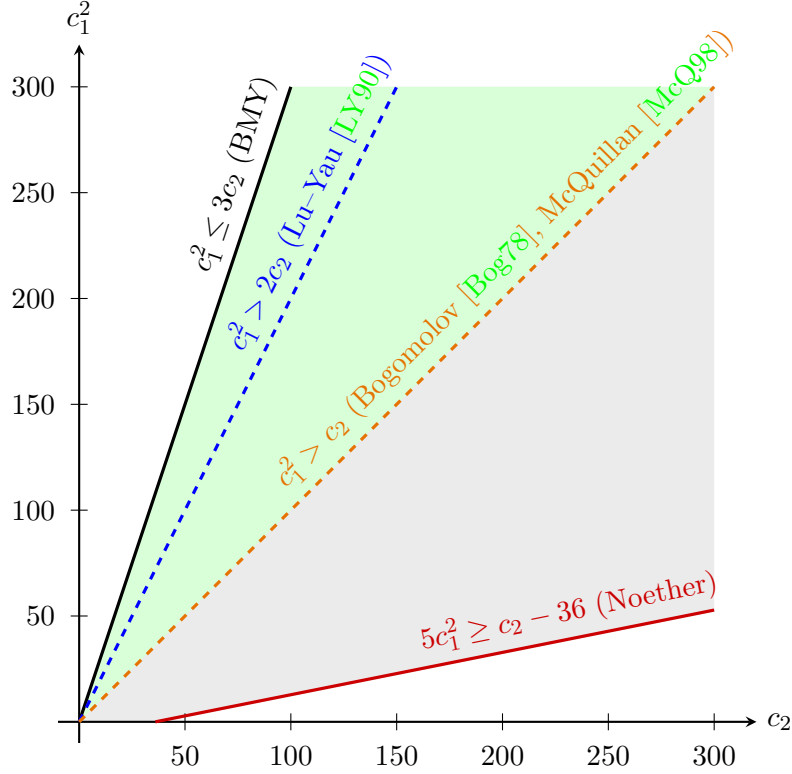

In a seminal series of papers in the 1970s \cite{Hor75, HorI, HorII}, Horikawa extensively studied these surfaces, classifying them via their canonical images as certain double covers and describing their moduli spaces (see \S\ref{sec:prelim-horikawa}).
Our primary focus in this article is on very general Horikawa surfaces.

We collect results on hyperbolicity of Horikawa surfaces by \cite{RoRo13, Liu18, CoRi19} in the following list. 
In what follows, $\ff_d$ denotes the $d$-th Hirzebruch surface, which comes with a natural projection $\ff_d\to \pp^1$. We denote by $\Gamma$ the class of a fiber of this projection and by $\Delta_0$ the class of the zero-section 
($\Gamma^2 = 0$, $\Delta_0 \cdot \Gamma = 1$ and $\Delta_0^2=-d$). 
When $d=0$, we maintain the same notation by letting $\Delta_0$ and $\Gamma$ denote the fiber classes of the two natural projections.

\begin{enumerate}
    \item A very general quintic surface in $\pp^3$ is Demailly algebraically hyperbolic \cite{CoRi19}. 
    
    \noindent ($c_1^2=5, c_2=55$, $p_g=4$, second kind, type I)
    \item Double covers of $\pp^2$ branched over a very general degree 10 curve are Demailly algebraically hyperbolic \cite{    RoRo13, Liu18}. 
    \noindent ($c_1^2=8, c_2=76$, $p_g=4$, first kind, type $(\infty)$)
    \item Double covers of $\pp^2$ branched over a very general degree 8 curve contain elliptic curves, but not rational curves \cite{RoRo13}. 
    \noindent ($c_1^2=2, c_2=46$, $p_g=3$, first kind, type $(\infty)$)
    \item For any integer $n\ge 3$ divisible by 3, double covers of $\ff_{\frac{n-3}{3}}$ branched over a very general curve of class $6\Delta_0+2n\Gamma$ contain elliptic curves, but not rational curves \cite{RoRo13}.
    
    \noindent ($c_1^2=2n-2, c_2=10n+26$, $p_g=n+1$, first kind, type $(d)$)
\end{enumerate}

Prior to this work, the \textit{pseudo}-hyperbolicity of the Horikawa surfaces in (3) and (4) remained an open question. 
One of the contributions of this article is proving that these surfaces contain only finitely many elliptic curves, which we are able to explicitly describe and enumerate.
We note that the aforementioned authors also proved hyperbolicity results for other cyclic covers of the projective plane and of Hirzebruch surfaces $\mathbb{F}_d$ that are not necessarily Horikawa surfaces; see \cite[Theorems 32 and 40]{RoRo13} and \cite[Theorems 1.2 and 1.3]{Liu18}.

We say that a smooth complex projective surface of general type is \emph{pseudo-Lang algebraically hyperbolic} if it contains only finitely many rational or elliptic curves, and we say that it is \emph{Lang algebraically hyperbolic} if it does not contain any rational or elliptic curves.
Our main result is the following:

\begin{theorem}\label{theorem:main}
{\em
A very general Horikawa surface  with $p_g\ge 5$ is pseudo-Lang algebraically hyperbolic.
Moreover, so is a very general Horikawa surface of the first kind with $p_g\in \{3,4\}.$
}
\end{theorem}

For a fixed geometric genus $p_g$, the Gieseker moduli space of Horikawa surfaces of either kind consists of one or two irreducible components.
A \textit{very general} Horikawa surface refers to a surface parametrized by a very general point  in its respective component, that is, a point outside a countable union of proper Zariski closed subsets.
We detail this in \S\ref{sec:horikawa-moduli}.

Each component of the moduli space is stratified by the \textit{type} of the Horikawa surface (see \S\ref{sec:prelim-horikawa}).
Our methods yield explicit data: for a very general surface in nearly every stratum (with the exception of a few lower-dimensional strata for the second kind), we completely characterize and enumerate all rational and elliptic curves.
For a summary of these curve counts, organized by geometric genus and stratum, see \S\ref{sec:summary-1st}, \S\ref{sec:summary-2nd} and Tables~\ref{table:horikawa-count-1} and \ref{table:horikawa-count-2}.
We note that we do not consider Horikawa surfaces of the second kind with $p_g \in \{2,3,4\}$ in this article.

As an illustration of our contributions and the flavor of our results, we highlight several cases that sharpen results (3) and (4) from \cite{RoRo13}.

\begin{enumerate}
    \item Double covers of $\pp^2$ branched over a very general degree 8 curve $D$ contain no rational curves and 1320 elliptic curves, arising from bitangent lines to $D$. 
    
    \noindent ($c_1^2=2, c_2=46$, $p_g=3$, first kind, type $(\infty)$, see Corollary~\ref{cor:p2-horikawa-cover})
    
    \item Double covers of $\ff_0\simeq\pp^1\times \pp^1$ branched over a very general curve $D\in |6\Delta_0+6\Gamma|$ contain no rational curves and 120 elliptic curves, arising from fibers of the natural projections $\ff_0\to \pp^1$ that are tangent to $D$. 

    \noindent ($c_1^2=4, c_2=56$, $p_g=4$, first kind, type $(0)$, see Corollary~\ref{cor:hirzebruch-f0-horikawa-cover})
    
    \item Fix an integer $n\ge 5$ and an integer $d<\frac{n-1}{3}$ such that $n-d$ is odd.
    Double covers of $\ff_d$ branched over a very general curve $D\in |6\Delta_0+(n+3+3d)\Gamma|$ contain no rational curves and $10n+30$ elliptic curves arising from curves of class $\Gamma$ that are tangent to $D$.
    
    \noindent($c_1^2=2n-2$, $c_2=10n+26$, $p_g=n+1\ge6$, first kind, type $(d)$, see Corollaries~\ref{cor:hirzebruch-horikawa-cover} and \ref{cor:hirzebruch-f0-horikawa-cover})
\end{enumerate}
Finally, we remark that, except for double covers of $\pp^2$ branched over a very general degree 10 curve, it remains open whether or not the surfaces addressed in Theorem~\ref{theorem:main} are pseudo-Demailly algebraically hyperbolic (see Definition~\ref{def:alg-hyp}).\\

\noindent \textbf{Proof ideas.}
By Horikawa's classification, the surfaces in question are birational to double covers $S'$ of certain smooth, regular rational surfaces $W$, branched over an explicitly described locus $D.$
$W$ is either the projective plane $\pp^2$, Hirzebruch surfaces $\ff_d$, or their blowups at one or two points.
Our starting point for studying the hyperbolicity of $S'$ is the framework utilized by Roulleau and Rousseau \cite{RoRo13}:
curves on the double cover $S'$ induce orbifold curves mapping to the orbifold pair $(W,\frac{1}{2}D)$. 
Consequently, one can bound the genus and degree of these curves by using results about the log algebraic hyperbolicity (see Definition~\ref{def:alg-hyp-log}) of the log pair $(W,D)$.

The aforementioned results in \cite{RoRo13} relied on Chen's log algebraic hyperbolicity inequalities on $(\pp^2,D)$ and $(\ff_d,D)$ for a very general basepoint-free $D$, which were established by degenerating $D$ into a union of components (see \cite{Che01}).
Chen's bounds were optimal in the sense that they detect exactly when such $(W,D)$ is strictly log algebraically hyperbolic or not. 
The primary innovations of this paper are more refined log algebraic hyperbolicity inequalities (e.g. Propositions \ref{prop:p2-horikawa} and \ref{prop:hirzebruch}), which may be of independent interest.
Building on the techniques of \cite{CRY22, ATY24, IMRY25}, we connect the log algebraic hyperbolicity of very general pairs $(W,D)$ to:
\begin{enumerate}
    \item the positivity of the log tangent bundle $T_W(-\log F)$ with respect to certain fixed components $F\subset D$, and
    \item section-dominating collections (see Definition~\ref{def:sec-dom}) of line bundles $L_i$ for the basepoint-free part $D'\subset D.$
\end{enumerate}
See Proposition~\ref{prop:log-normal-degree-bound} and Remark~\ref{rem:log-normal-degree-bound}.

Crucially, this approach sometimes allows us to explicitly characterize the ``exceptional'' curves that violate these inequalities and the ``extreme'' curves that achieve strict equality. 
Specifically,  equality in our bounds sometimes forces the restricted kernel bundle $M_{L_i}\vert_C$ to admit global sections, which forces $C$ to lie within some section of $L_i$.
We utilize this improvement on $\pp^2$ and $\ff_d$, and we extend these methods to analyze blowups of $\ff_d$ at one or two points on the same fiber.\\

\noindent \textbf{Outline.} In \S\ref{sec:prelim-horikawa}, we review the preliminaries concerning Horikawa surfaces and their moduli spaces. 
In \S\ref{sec:prelim-hyp}, we discuss the hyperbolicity of log pairs, focusing on surface pairs $(W,D)$ where $D$ is a very general divisor. 
We formalize the methods outlined above,
presenting both the variational argument used to study the log algebraic hyperbolicity of $(W,D)$ (see Proposition~\ref{prop:log-normal-degree-bound}) and how to lift this to the hyperbolicity of the double cover $S'$ (see Lemma~\ref{lemma:log-hyp-to-cover-hyp}).
Finally, in \S\ref{sec:horikawa-hyp}, we apply this framework to the specific pairs $(W,D)$ and double covers $S'$ that give rise to very general Horikawa surfaces.\\

\noindent\textbf{Acknowledgements.}
This project was initiated during the 2025 Junior Algebraic Geometry Workshop held at the University of California, Los Angeles.
The authors thank Joaqu\'in Moraga for suggesting the problem and for many helpful discussions and comments, and Burt Totaro, Giancarlo Urz\'ua and Juan Pablo Z\'u\~niga for their valuable comments on an earlier draft.
WY was supported by an AMS-Simons Travel Grant. 
AA was partially supported by NSF grant DMS-2401387 (PI Jenia Tevelev). 

\section{Preliminaries on Horikawa surfaces}
\label{sec:prelim-horikawa}

In this section, we summarize the classification and moduli of Horikawa surfaces, drawing primarily from the foundational work of Horikawa \cite{Hor75, HorI, HorII}. 
Specifically, we describe these surfaces of a given geometric genus $p_g$ as double covers of certain rational surfaces, giving the necessary and sufficient conditions on the linear system and the singularities of their branch loci. 
We also review the moduli space of Horikawa surfaces of a given geometric genus $p_g$ and characterize the general point in each component. 
While we do not claim originality for the results in this section, our goal is to present Horikawa's classification in a streamlined manner that highlights precisely where our new contributions apply.

\subsection{Horikawa surfaces}
\label{sec:horikawa-classification}
It is well-known that minimal smooth complex algebraic surfaces $S$ of general type must satisfy Noether's inequality $c_{1}^2\geq 2p_g-4$. 
In the articles cited above, Horikawa classified all the surfaces of general type satisfying $c_1^2=2p_g-4$, resp. $c_1^2=2p_g-3$, which we refer to as \emph{Horikawa surfaces of the first, resp. second, kind}. 

\subsubsection{Horikawa surfaces of the first kind}
Following the presentation in \cite{HorI}, the numerical invariants $p_g,c_1^2$ and $c_2$ of a Horikawa surface $S$ of the first kind are determined by an integer $n\ge 2$ as follows:
\begin{itemize}
    \item $p_g=n+1$,
    \item $c_1^2=2p_g-4=2n-2$, and
    \item $c_2=5c_1^2+36$.
\end{itemize}

For these surfaces, the linear system $|K_S|$ is basepoint-free. 
$S$ arises as the minimal resolution of singularity of a double cover $S'$ of a rational surface $W$, which is either $\pp^2$ or a Hirzebruch surface $\ff_d$ for $d\ge 0$, branched over a reduced divisor $D$ that is said to have \emph{no infinitely near triple points}, defined as follows.

\begin{definition}\label{def:no_inf_near_triple_points} \cite{HorI} 
{\em
Let $W$ be a smooth projective surface. We say that a reduced effective divisor $D\subset W$ has \emph{no infinitely near triple points }if the following conditions are satisfied:
\begin{enumerate}
    \item $D$ has no singular points of multiplicity $4$, and
    \item every triple point $s$ of $D$  (if any) decomposes into a singularity of multiplicity $\leq 2$ after a blow-up with center at $s$.
\end{enumerate}
}
\end{definition}

In this article, unless otherwise noted, the branch loci $D$ are always reduced divisors that do not have infinitely near triple points.

When $n=2$, $S'$ can be taken as a double cover of $\pp^2$ branched over a divisor $D\in |8H|$.
When $n=5$, depending on the canonical image of $S$, $S'$ may also be a double cover of $\pp^2$, in this case it is branched over a divisor $D\in |10H|$.
Horikawa referred to these two families of surfaces as being of type $(\infty)$.

When $n \ge 3$, the surface $S'$ may be a double cover of a Hirzebruch surface $\mathbb{F}_d$, where $0 \le d \le \min\{n-3, \frac{n+3}{2}\}$ and $n - d$ is odd. 
In this case, the branch divisor $D$ is in the linear system $|6\Delta_0+(n+3+3d)\Gamma|$. 
These are referred to as being of type $(d)$. 
A general element of the linear system $|6\Delta_0 + (n + 3 + 3d)\Gamma|$ is smooth whenever $d \le \frac{n+3}{3}$. 
Outside of that range, i.e. whenever $\frac{n+3}{3} < d \le \frac{n+3}{2}$, a general element is reducible and takes the form $\Delta_0 \cup D'$, where $D'$ is a general, smooth element of $|5\Delta_0 + (n + 3 + 3d)\Gamma|$. 
In this latter case, $\Delta_0$ intersects $D'$ transversely at $n + 3 - 2d$ points; in particular, they are disjoint when $d = \frac{n+3}{2}$.

When $n=3,4$ or $5$, $S'$ may be a double cover of $\ff_{n-1}$ branched over a divisor $D$ in $|6\Delta_0+4n\Gamma|$.
These surfaces are said to be of type $(2'),(3')$ or $(4')$, respectively.
A general element of $|6\Delta_0+12\Gamma|$ is smooth, while a general element of $|6\Delta_0+4n\Gamma|$ with $n=4$ or $5$ is reducible of the form $\Delta_0\cup D'$ where $D'$ is a general, smooth element of $|5\Delta_0+4n\Gamma|.$ 
These two components intersect at one point when $n=4$ while they are disjoint when $n=5.$

This is the exhaustive classification. In particular, when $n\ge 6$, all Horikawa surfaces of the first kind fall into type $(d)$.

\begin{table}[htbp]
    \centering
    \caption{Summary of classification of Horikawa surfaces of the first kind.}
    \label{table:horikawa1}
    \small
    \begin{tabularx}{\textwidth}{l |l l l l l}
        \toprule
        \textbf{$n$} & \textbf{Type} & \textbf{$W$} & \textbf{Linear system of $D$} & \textbf{General $D$} & \textbf{$\Delta_0 \cdot D'$} \\ 
        \midrule
        2 & $(\infty)$ & $\mathbb{P}^2$ & $|8H|$ & Smooth & - \\
        \midrule
        5 & $(\infty)$ & $\mathbb{P}^2$ & $|10H|$ & Smooth & - \\
        \midrule
        \multirow{3}{*}{$\ge3$} & \multirow{3}{*}{\parbox{2.5cm}{$(d)$, with $n-d$ odd and $0\le d \le \min\{n-3, \frac{n+3}{2}\}$}} & \multirow{3}{*}{$\mathbb{F}_d$} & \multirow{3}{*}{$|6\Delta_0+(n+3+3d)\Gamma|$} & Smooth (if $d \le \frac{n+3}{3}$) & - \\
         & & & & $\Delta_0\cup D'$ (if $d > \frac{n+3}{3}$) & $n+3-2d$ \\
         &\\
        \midrule
        \multirow{3}{*}{$3, 4, 5$} & \multirow{3}{*}{$(2'), (3'), (4')$} & \multirow{3}{*}{$\mathbb{F}_{n-1}$} & \multirow{3}{*}{$|6\Delta_0+4n\Gamma|$} & Smooth ($n=3$) & - \\
         & & & & $\Delta_0\cup D'$ ($n=4$) & $1$ \\
         & & & & $\Delta_0\sqcup D'$ ($n=5$) & $0$ \\
        \bottomrule
    \end{tabularx}
\end{table}

\begin{remark}
For surfaces of type $(d)$ with $d > \frac{n+3}{3}$ and a general component $D'$, the transverse intersection of $\Delta_0$ and $D'$ implies that the double cover $S'$ branched over $D = \Delta_0 \cup D'$ possesses an $A_1$ singularity over each intersection point. Locally, each intersection point corresponds to a singularity of the form $xy = z^2$, which is analytically equivalent to a standard $A_1$ singularity.
\end{remark}

\begin{remark}
Among these cases, there is a distinguished one in terms of moduli: when $D'$ and $\Delta_0$ are disjoint, which only occurs when $d=\frac{n+3}{2}$.
In this setting, the condition that $n-d$ is odd implies that $(2d-3)-d=d-3$ is odd, so $d$ must be even and there is an integer $k$ such that $d=2k+2$.  
Since $c_{1}^2=2n-2=2(2d-3)=4d-8=8k$, we conclude that this particular case can only occur when $c_1^2$ is a multiple of $8$.    
\end{remark}

\subsubsection{Horikawa surfaces of the second kind \cite{HorII}}
The numerical invariants $p_g,c_1^2$ and $c_2$ of a Horikawa surface $S$ of the second kind is determined by an integer $n\ge 1$ as follows:
\begin{itemize}
    \item $p_g=n+1$,
    \item $c_1^2=2p_g-3=2n-1$, and
    \item $c_2=5c_1^2+30$.
\end{itemize}

In this article, we focus on the case $n \ge 4$, where $|K_S|$ possesses a unique base point. 
Let $\pi: \widetilde{S} \to S$ denote the blowup of $S$ at this base point. 
Then $\widetilde{S}$ is the minimal resolution of singularities of a double cover $S'$ of a rational surface $W$, which is either a Hirzebruch surface $\mathbb{F}_d$ or a blowup of $\mathbb{F}_d$ at one or two points lying on the same fiber $\Gamma_0$, and the branch locus $D \subset W$ is a reduced divisor with no infinitely near triple points (see Definition~\ref{def:no_inf_near_triple_points}).
We describe the details in the following paragraphs.
For the classification in the low geometric genus cases, i.e. when $n=1,2,3$, please refer to \cite{Hor75} and \cite[Section 2]{HorII}.

We formalize the blowup of $\mathbb{F}_d$ at two points $x$ and $y$ lying on the same fiber $\Gamma_0$ as follows. 
Let $q_1:W_1\to \ff_d$ be the blowup of $\ff_d$ at $x \in \Gamma_0$, and let $\Gamma_1$ be the proper transform of $\Gamma_0$. 
We then define $q_2: W \to W_1$ as the blowup at a second point $y \in \Gamma_1$. 
If $y$ lies on the intersection of $\Gamma_1$ and the exceptional divisor of $q_1$ (over $x$), we say $y$ is \emph{infinitely near} to $x$.
We denote the composed map by $q = q_1 \circ q_2: W\to \ff_d$. 
Let $\widetilde{\Gamma}$ be the proper transform of $\Gamma_0$, and let $E_x$ and $E_y$ be the exceptional curves over $x$ and $y$ under $q$. 
By a slight abuse of notation, when $d>0$, we write $y \in \Delta_0$ if $y$ lies on the full transform of $\Delta_0$ under $q_1$. 

\begin{figure}[htbp]
    \centering
    \begin{tikzpicture}[scale=1.2]
        
        \begin{scope}[shift={(0,0)}]
            \node at (1, 2.7) {$W$};
            \draw[thick] (0,0) -- (2,0) node[right] {};
            \draw[thick, blue] (0.3, 1.8) -- (1.5, 1.2) node[above right] {$E_x$};
            \draw[thick, red] (0.3, 0.9) -- (1.5, 0.3) node[above right] {$E_y$};
            \draw[thick] (1,-0.5) -- (1,2) node[above left] {$\widetilde{\Gamma}$};
        \end{scope}

        \draw[->, thick] (2.8, 1) -- (3.8, 1) node[midway, above] {$q_2$};

        \begin{scope}[shift={(4.5,0)}]
            \node at (1, 2.7) {$W_1$};
            \draw[thick] (0,0) -- (2,0) node[right] {};
            \draw[thick, blue] (0.3, 1.8) -- (1.5, 1.2) node[above right] {};
            \draw[thick] (1,-0.5) -- (1,2) node[above left] {$\Gamma_1$};
            \filldraw (1,0.7) circle (1.5pt) node[right] {$y$};
        \end{scope}

        \draw[->, thick] (7.3, 1) -- (8.3, 1) node[midway, above] {$q_1$};

        \begin{scope}[shift={(9,0)}]
            \node at (1, 2.7) {$\mathbb{F}_d$};
            \draw[thick] (0,0) -- (2,0) node[right] {$\Delta_0$};
            \draw[thick] (1,-0.5) -- (1,2) node[above left] {$\Gamma_0$};
            \filldraw (1,1.5) circle (1.5pt) node[right] {$x$};
        \end{scope}

    \end{tikzpicture}
    \caption{The sequence of blowups $q=q_2\circ q_1$ for generic choices of $x\in \Gamma_0$ and $y\in \Gamma_1$.}
    \label{fig:blowup_sequence}
\end{figure}
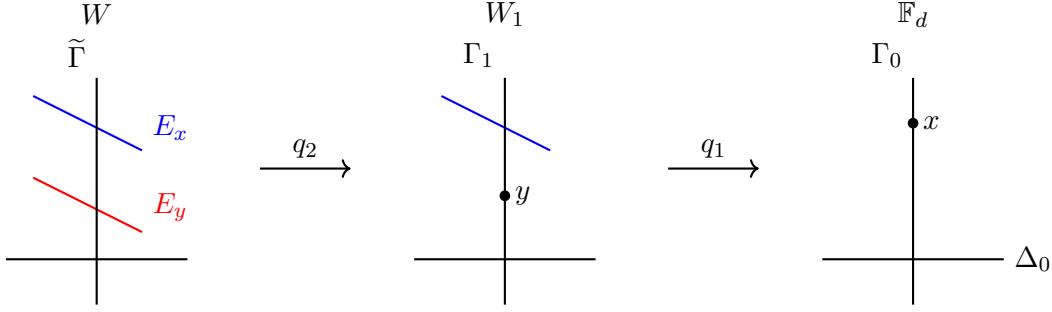

When $n \ge 4$, the surface $S'$ may be a double cover of the blowup $W$ of $\mathbb{F}_d$ at two points $x,y\in\Gamma_0$, as described in the previous paragraph.
The integer $d$ and the points $x,y$ are subject to the global constraints $d\le n-3$ and $n-d$ odd, and must fall into one of the following four cases:
\begin{enumerate}
    \item  $d=0$. There are no further conditions on $x$ and $y$. (Note that here we use $\Delta_0$ to denote a fiber of the opposite fibration, and it is not unique.)
    \item $0<d\le \frac{n+4}{3}$. Neither $x$ nor $y$ lies on $\Delta_0$, and they may or may not be infinitely near.
    \item $0<d\le \frac{n+1}{2}$. $x \in \Delta_0$, and $x$ and $y$ are infinitely near.
    \item $0<d\le \frac{n+2}{2}$. $x \in \Delta_0$ and $y \notin \Delta_0$, hence they are not infinitely near.
\end{enumerate}

These surfaces are referred to as being of type $(d)$. 
In all such cases, the branch divisor $D$ belongs to the linear system $|6q^*\Delta_0 + (n+5+3d)q^*\Gamma - 4E_x - 4E_y|$. 
Assuming $x$ and $y$ are not infinitely near, the structure of a general element in this system depends on $d$:
\begin{itemize}
\item When $d \le \frac{n+4}{3}$ and neither $x$ nor $y$ lies on $\Delta_0$ if $d>0$, a general $D$ is a disjoint union $\widetilde{\Gamma} \cup D'$, where $D'$ is a general, smooth element of $|6q^*\Delta_0 + (n+4+3d)q^*\Gamma - 3E_x - 3E_y|$.
\item When $\frac{n+4}{3} < d \le \frac{n+2}{2}$, which forces $x\in \Delta_0$ and $y\notin \Delta_0$, a general $D$ is a union $\widetilde{\Gamma} \cup \widetilde{\Delta}_0 \cup D'$, where $D'$ is a general, smooth element of $|5q^*\Delta_0 + (n+4+3d)q^*\Gamma - 2E_x - 3E_y|$. 
In this case, $\widetilde{\Gamma}$ is disjoint from the other two components, and $\widetilde{\Delta}_0$ intersects $D'$ transversely at $n + 2 - 2d$ points (becoming completely disjoint when $d = \frac{n+2}{2}$).
\end{itemize}
For these surfaces of type $(d)$, the map $\pi$ contracts the proper transform of $\widetilde{\Gamma}$.

When $n=5$, $S'$ may also be a double cover of $\mathbb{F}_2$ branched over a divisor $D \in |8\Delta_0+14\Gamma|$. 
A general such $D$ is a disjoint union $\Delta_0 \cup D'$, with $D'$ a general, smooth element of $|7\Delta_0+14\Gamma|$. 
For these surfaces, referred to as being of type $(2^*)$, $\pi$ contracts the proper transform of $\Delta_0$.

When $n=4$, there are two additional cases:
\begin{itemize}
\item \emph{Type $(1^*)$:} 
The surface $S'$ is a double cover of the blowup of $\mathbb{F}_1$ at a single point $x \in \Delta_0$, branched over a divisor $D \in |8q^*\Delta_0+10q^*\Gamma-4E_x|$. 
A general $D$ is a disjoint union $\widetilde{\Delta}_0 \cup D'$, where $D' \in |7q^*\Delta_0+10q^*\Gamma-3E_x|$ is smooth and general. 
Here, $\pi$ contracts the proper transform of $\widetilde{\Delta}_0$.
\item \emph{Type $(3')$:} The surface $S'$ is a double cover of the blowup of $\mathbb{F}_3$ at two points $x \in \Delta_0$ and $y \notin \Delta_0$ not infinitely near, branched over $D \in |6q^*\Delta_0+18q^*\Gamma-4E_x-4E_y|$. 
Similar to some type $(d)$ cases, a general $D$ is a disjoint union $\widetilde{\Gamma} \cup \widetilde{\Delta}_0 \cup D'$, where $D'$ is a general, smooth element of $|5q^*\Delta_0+17q^*\Gamma-2E_x-3E_y|$. 
Here, $\pi$ contracts the proper transform of $\widetilde{\Gamma}$.
\end{itemize}

This gives the exhaustive classification for $n \ge 4$. When $n \ge 6$, all Horikawa surfaces of the second kind are of type $(d)$.

\begin{table}[htbp]
    \centering
    \caption{Summary of classification of Horikawa surfaces of the second kind for $n\ge 4$.}
    \label{table:horikawa2}
    \small
    \begin{tabularx}{\textwidth}{l |l l l l}
        \toprule
        \textbf{$n$} & \textbf{Type} & \textbf{$W$} & \textbf{Linear system of $D$} & \textbf{General $D$} \\ 
        \midrule
        \multirow{3}{*}{$\ge4$} & \multirow{3}{*}{\parbox{2.5cm}{$(d)$, with $n-d$ odd and $0\le d \le \min\{n-3, \frac{n+2}{2}\}$}} & \multirow{3}{*}{$\operatorname{Bl}_{x,y}\mathbb{F}_d$} & \multirow{3}{*}{\parbox{3.5cm}{$|6q^*\Delta_0 + (n+5+3d)q^*\Gamma\\ - 4E_x - 4E_y|$}} & $\widetilde{\Gamma}\sqcup D'$ (if $d \le \frac{n+4}{3}$)  \\
         & & & & $\widetilde{\Gamma}\sqcup (\widetilde{\Delta}_0\cup D')$ and $\widetilde{\Delta}_0\cdot D'=n+2-2d$ \\ & & & & \hspace{1.5in}(if $d > \frac{n+4}{3}$)\\
        \midrule
        5 & $(2^*)$ & $\ff_2$ & $|8\Delta_0+14\Gamma|$ & $\Delta_0\sqcup D'$ \\
        \midrule
        4 & $(1^*)$ & $\operatorname{Bl}_x\ff_1$ & $|8q^*\Delta_0+10q^*\Gamma-4E_x|$ & $\widetilde{\Delta}_0\sqcup D'$ \\
        \midrule
        4 & $(3')$ & $\operatorname{Bl}_{x,y}\mathbb{F}_3$ & $|6q^*\Delta_0+18q^*\Gamma-4E_x-4E_y|$ & $\widetilde{\Gamma}\sqcup \widetilde{\Delta}_0\sqcup D'$\\
        \bottomrule
    \end{tabularx}
\end{table}

\subsection{Moduli of Horikawa surfaces}
\label{sec:horikawa-moduli}
Recall that if $S$ is a minimal smooth complex algebraic surface of general type, then $K_S$ is nef and $K_{S}^2>0$. 
The Hodge index theorem and the adjunction formula imply that the only irreducible curves preventing $K_S$ from being ample are $(-2)$-curves. 
The number of $(-2)$-curves on $S$ is finite and it is at most $\rho(S)-1$, where $\rho(S)$ is the Picard number of $S$.
On the subspace of $H_2(S,\mathbb{Q})$ spanned by these $(-2)$-curves, the intersection form is negative definite.
After contracting all $(-2)$-curves in $S$, we obtain a normal surface $S_{\operatorname{can}}$ such that the canonical divisor $K_{S_{\operatorname{can}}}$ is ample. 
Each connected component of the union of $(-2)$-curves is contracted to a Du Val singularity in $S_{\operatorname{can}}$ and the surface $S_{\operatorname{can}}$ is called the \emph{canonical model} of $S$. 

For fixed positive integers $c_1^2$ and $c_2$, Gieseker \cite{MR498596} proved that there is a quasi-projective coarse moduli space $\mathcal{M}_{c_1^2,\chi(\mathcal{O)}}$ parametrizing isomorphism classes of these canonical models (which correspond to isomophism classes of the minimal surfaces $S$) with $c^2_1(S)=c_1^2$ and $c_2(S)=c_2$. 
These invariants are related via Noether's formula: $12\chi(\mathcal{O}_S)=c_1^2+c_2.$

Following the notation in \cite{CiPa25}, for $n\ge3$ we let $\mathcal{M}^{\operatorname{Hor},1}_{n}$ be the loci in the Gieseker moduli space $\mathcal{M}_{2n-2,n+2}$ of the Horikawa surfaces of the first kind of type $(d)$. 
When $n=2$, all Horikawa surfaces of the first kind are of type $(\infty)$, and $\mathcal{M}_{2,4}$ is irreducible of dimension $\dim \vert{}\mathcal{O}_{\mathbb{P}^2}(8)\vert{}-\dim \operatorname{PGL}_3 = 44-8=36$.
For $n\ge 4,$ we denote by $\mathcal{M}^{\operatorname{Hor},2}_{n}$ the loci in the Gieseker moduli space $\mathcal{M}_{2n-1,n+2}$ of the Horikawa surfaces of the second kind of type $(d)$. 

In terms of moduli, Horikawa's results from \cite{HorI, HorII} can be summarized in Table~\ref{table:horikawa-moduli} and the following two theorems which focus on the moduli of type $(d)$ Horikawa surfaces.
We say that a property holds for a \textit{general} point of a moduli space $M$ if it holds on some nonempty Zariski open subset of $M$, and a property holds for a \textit{very general} point of $M$ if it holds outside a countable union of proper closed subsets.

\begin{theorem}[Moduli of Horikawa surfaces of the first kind, see Theorem 3.3 in \cite{CiPa25}]
\label{theorem:moduli-1st}
{\em
For a fixed $n\ge 3$,
    \begin{enumerate}
    \item $\mathcal{M}^{\operatorname{Hor},1}_{n}$ is a union of irreducible components of $\mathcal{M}_{2n-2, n+2}$; for $n \ge 6$, it coincides exactly with $\mathcal{M}_{2n-2, n+2}$.
    \item If $4 \nmid n-1$ (or equivalently, $8 \nmid c_1^2$), then $\mathcal{M}^{\operatorname{Hor},1}_{n}$ is irreducible of dimension $7n+21$. For $n \ge 6$, a general point in this moduli space is of type $(0)$ or $(1)$ whose branch locus is a smooth, irreducible curve that is general in its linear system.

\item If $n-1 = 4k$ with $k \ge 2$, then $\mathcal{M}^{\operatorname{Hor},1}_{n}$ consists of two disjoint irreducible components, each of dimension $7n+21$:
\begin{enumerate}
    \item A component $\mathcal{M}^{\operatorname{Hor},1a}_{4k+1}$ parametrizing surfaces of type $(d)$ with $d$ even and $0 \le d \le 2k$. A general point in this component is of type $(0)$, branched over a smooth, irreducible curve.
    \item A component $\mathcal{M}^{\operatorname{Hor},1b}_{4k+1}$ parametrizing surfaces of type $(2k+2)$. A general point in this component is of type $(2k+2)$, branched over a disjoint union of two smooth, irreducible curves.
\end{enumerate}
In both cases, for a general point, the branch loci are general in their respective linear systems.
    \end{enumerate}
}
\end{theorem}

\begin{theorem}[Moduli of Horikawa surfaces of the second kind for $n\ge 4$, see Theorem 3.5 in \cite{CiPa25}]
\label{theorem:moduli-2nd}
{\em
For a fixed $n\ge 4$,
    \begin{enumerate}
        \item $\mathcal{M}^{\operatorname{Hor},2}_{n}$ is a union of irreducible components of $\mathcal{M}_{2n-1, n+2}$; for $n \ge 6$, it coincides exactly with $\mathcal{M}_{2n-1, n+2}$.
        \item If $4 \nmid n$, then $\mathcal{M}^{\operatorname{Hor},2}_{n}$ is irreducible of dimension $7n+19$. For $n \ge 6$, a general point in this moduli space is of type $(0)$ or $(1)$ where the centers $x$ and $y$ are general points on the same fiber $\Gamma_0$ (hence not infinitely near, and neither lies on $\Delta_0$ if $d>0$). Its branch locus is a disjoint union of two smooth, irreducible curves that are general in their linear systems.

\item If $n = 4k$ with $k \ge 2$, then $\mathcal{M}^{\operatorname{Hor},2}_{n}$ consists of two disjoint irreducible components, each of dimension $7n+19$:
\begin{enumerate}
    \item A component $\mathcal{M}^{\operatorname{Hor},2a}_{4k}$ parametrizing surfaces of type $(d)$ with $d$ odd and $1 \le d \le 2k-1$. A general point in this component is of type $(1)$, where $x$ and $y$ are general points on $\Gamma_0$ (not infinitely near and neither on $\Delta_0$). The branch locus is a disjoint union of two smooth, irreducible curves.
    \item A component $\mathcal{M}^{\operatorname{Hor},2b}_{4k}$ parametrizing surfaces of type $(2k+1)$. A general point in this component is of type $(2k+1)$, where $x \in \Delta_0$ and $y \notin \Delta_0$ are points on $\Gamma_0$ (not infinitely near and $y\in \Gamma_0$ is general). The branch locus is a disjoint union of three smooth, irreducible curves.
\end{enumerate}
In both cases, for a general point, the branch loci are general in their respective linear systems.
    \end{enumerate}
}
\end{theorem}

In parts (2) and (3a) of the above theorems, a Horikawa surface of type $(d)$ deforms into type $(d-2)$, and eventually into type $(0)$ or $(1)$ depending on whether $d$ is even or odd.

Recently, there has also been significant progress in explicitly describing the boundary surfaces in the KSBA compactification of the moduli space of Horikawa surfaces (see e.g. \cite{MNU25, CiPa25, AEHK25}).

\begin{table}[htbp]
    \centering
    \caption{Deformation families of Horikawa surfaces of the first and second kind. The leftmost type in each chain represents the general type in its respective family. $A \leftarrow B$ indicates that $B$ deforms into $A$.}
    \label{table:horikawa-moduli}
    \renewcommand{\arraystretch}{1.3} 
    \begin{tabular}{@{}llcl@{}}
        \toprule
         & \textbf{Condition on $n$} & \textbf{Families} & \textbf{Deformation chains (Left = General)} \\
        \midrule
        
        \multirow{7}{*}{\textbf{First kind}} 
        & $n=2$ & 1 & $(\infty)$ \\
        \cmidrule{2-4}
        & $n=3$ & 1 & $(0) \leftarrow (2')$ \\
        \cmidrule{2-4}
        & $n=4$ & 1 & $(1) \leftarrow (3')$ \\
        \cmidrule{2-4}
        & $n=5$ & 2 & $(0) \leftarrow (2)$ \\
        & & & $(\infty) \leftarrow (4')$ \\
        \cmidrule{2-4}
        & $n \ge 6$, $4 \nmid n-1$ & 1 & $(0) \text{ or } (1) \leftarrow \dots \leftarrow (d)$ \\
        \cmidrule{2-4}
        & $n = 4k+1 \ge 9$ & \multirow{2}{*}{2} & $(0) \leftarrow (2) \leftarrow \dots \leftarrow (2k)$ \\
        & (i.e. $4 \mid n-1$) & & $(2k+2) = \left(\frac{n+3}{2}\right)$  \\
        
        \midrule
        \midrule
        
        \multirow{7}{*}{\textbf{Second kind ($n\ge 4)$}} 
        & $n=4$ & 2 & $(1)$ \\
        & & & $(1^*) \leftarrow (3')$ \\
        \cmidrule{2-4}
        & $n=5$ & 2 & $(0) \leftarrow (2)$ \\
        & & & $(2^*)$ \\
        \cmidrule{2-4}
        & $n \ge 6$, $4 \nmid n$ & 1 & $(0) \text{ or } (1) \leftarrow \dots \leftarrow (d)$ \\
        \cmidrule{2-4}
        & $n = 4k \ge 8$ & \multirow{2}{*}{2} & $(1) \leftarrow (3) \leftarrow \dots \leftarrow (2k-1)$ \\
        & (i.e. $4 \mid n$) & & $(2k+1) = \left(\frac{n+2}{2}\right)$  \\
        
        \bottomrule
    \end{tabular}
\end{table}

\section{Preliminaries on hyperbolicity}\label{sec:prelim-hyp}

We cover the preliminaries necessary to prove hyperbolicity results for certain double covers $S'$ of rational surfaces $W$ branched over very general divisors $D$.
First, we summarize and compare different notions of hyperbolicity with a focus on algebraic surfaces.
Next, we give a brief exposition of orbifolds, as curves on the double cover $S'$ induce orbifold curves in the orbifold pair $(W,\frac{1}{2}D)$, as noted in \cite{RoRo13}.
Then, we establish results on the algebraic hyperbolicity of the log pair $(W, D)$. 
This requires introducing logarithmic sheaves, the ``variational argument'' used to study hyperbolicity when the branch locus $D$ is very general, and the associated kernel bundles. 
This latter part of our approach is adapted from techniques developed in \cite{CoRi23, CRY22, ATY24, IMRY25}.
Finally, in Section~\ref{sec:horikawa-hyp}, we apply this machinery to very general Horikawa surfaces.

\subsection{Algebraic hyperbolicity}

As touched upon in the introduction, Lang conjectured that a smooth complex projective variety $X$ is Brody hyperbolic if and only if every integral subvariety of $X$ has a desingularization of general type. 
This latter property can be viewed as an algebraic analogue of hyperbolicity, and we will refer to this property as \textit{Lang algebraic hyperbolicity.} 
Demailly \cite{Dem97} subsequently introduced an alternative algebraic notion of hyperbolicity, which he showed is weaker than Brody and Kobayashi hyperbolicity for projective varieties.
We also include the pseudo-version of this property.

\begin{definition}\label{def:alg-hyp}
{\em
A smooth complex projective variety $X$ is \emph{Demailly algebraically hyperbolic} if there is an ample line bundle $L$ on $X$ and a constant $\varepsilon>0$ such that every morphism $f: C \to X$ from a smooth projective curve $C$ that is birational onto its image satisfies
\begin{equation}
\label{eq:hyperbolicity-definition}
2g(C)-2\geq \varepsilon \deg f^*L,
\end{equation}
where $g(C)$ denotes the geometric genus of $C$.

$X$ is \emph{pseudo-Demailly algebraically hyperbolic outside of a proper closed subset $Z\subsetneq X$} if there is an ample line bundle $L$ and  a constant $\varepsilon > 0$ such that every map $f: C \to X$ as above with $f(C)\not\subset Z$ satisfies \eqref{eq:hyperbolicity-definition}.
}
\end{definition}

Analogous pseudo-variants can be defined for Brody hyperbolicity and other notions of hyperbolicity, which are relevant in light of Conjecture~\ref{conj:ggl} and its corresponding formulations for different notions of hyperbolicity.
These definitions remain well-defined when $X$ admits mild singularities, such as ordinary double points. 
Some double covers $S'$ that we study are branched over simple normal crossing (snc) divisors $D\subset W$, so they will admit these singularities over the nodes of $D$.

One observes that Demailly algebraically hyperbolic varieties do not contain any rational or elliptic curves, while
\textit{pseudo}-Demailly algebraically hyperbolic varieties $X$ may contain such curves, but they are strictly constrained to lie within the fixed proper closed subset $Z\subset X$. 
Consequently, a pseudo-Demailly algebraically hyperbolic \textit{surface} $S$ contains at most finitely many rational or elliptic curves; it is therefore pseudo-Lang algebraically hyperbolic provided that $S$ itself is of general type.
Except in the case of curves, it is not clear whether Demailly and Lang algebraic hyperbolicity are equivalent (see Proposition~\ref{prop:dem-implies-lang-surface} for one implication in the case of surfaces).
For a recent survey on the above and other notions of hyperbolicity, see e.g. \cite{Jav20}.

\begin{proposition}\label{prop:dem-implies-lang-surface}
    {\em If a smooth projective surface is Demailly algebraically hyperbolic, then it is also Lang algebraically hyperbolic.}
\end{proposition}

\begin{proof}
We show that all Demailly algebraically hyperbolic smooth surfaces are of general type.
Since Demailly algebraically hyperbolic surfaces cannot contain rational curves or elliptic curves, it is enough to check the condition in a minimal model. 
Let $X$ be a minimal smooth complex projective surface and denote by $\kappa(X)$ its Kodaira dimension. 
We discard all the cases with $\kappa(X)\neq 2$ via the Enriques--Kodaira classification of complex surfaces, see \cite[\S\S  5,6]{MR2030225}. 
If $\kappa(X)=-\infty$, then $X$ is uniruled.
If $\kappa(X)=1$, then $X$ is an elliptic surface, so it is covered by elliptic curves. 
If $\kappa(X)=0$, we have four cases. 
If $X$ is a K3 surface, then it contains a rational curve by a theorem of Bogomolov and Mumford. 
If $X$ is an Enriques surface, then it is covered by a K3 surface via a degree two unramified covering, so it contains a rational curve. 
If $X$ is a bielliptic surface, then $X$ is a product of elliptic curves modulo a finite group, so the projections yield elliptic curves in $X$.  

Finally, consider an abelian surface $X.$
Let $L$ be an ample line bundle on $X$, and let $f\colon C\to X$ be any nonconstant map.
Consider $f_n\colon C\to X$ given by $f_n=[n]\circ f$ where $[n]$ is the multiplication by $n$ map $[n]\colon X\to X$. 
Decomposing $L$ into its symmetric part $L_{s}$ and anti-symmetric part $L_{a}$, we have $L^{\otimes 2}\cong L_{s}\otimes L_{a}$. 
By the Theorem of the Cube, we have $[n]^{*}L_{s}\cong L_{s}^{\otimes n^2}$ and $[n]^{*}L_{a}\cong L_{a}^{\otimes n}$, which yields
$$\deg(f^{*}_nL)=\frac{1}{2}\left(n^2\deg(f^{*}L_s)+n\deg(f^{*}L_a)\right).$$
Since $L$ is ample, $\deg(f^{*}L_s)$ is positive. 
As $n\to \infty$, we have $\deg(f_n^{*}L)\to \infty$ and so
$$\frac{2g(C)-2}{deg(f_n^{*}L)}\to 0.$$
Therefore $X$ cannot be Demailly algebraically hyperbolic.
\end{proof}

We include the analogous definitions for log smooth pairs $(X,D)$, i.e. when $X$ is a smooth projective variety and $D$ a simple normal crossing divisor.
The following definition was first introduced by Chen in \cite{Che01}.

\begin{definition}\label{def:alg-hyp-log}
{\em
A log smooth pair $(X,D)$ is \emph{log (Demailly) algebraically hyperbolic} if there is an ample line bundle $L$ on $X$ and a constant $\varepsilon>0$ such that every morphism $f: C \to X$ from a smooth projective curve $C$ that is birational onto its image with $f(C)\not\subset D$ satisfies
\begin{equation}
\label{eq:hyperbolicity-definition-log}
2g(C)-2 + |f^{-1}(D)|\geq \varepsilon \deg f^*L,
\end{equation} 
where $g(C)$ denotes the geometric genus of $C$.

$(X,D)$ is \emph{pseudo-log (Demailly) algebraically hyperbolic outside of a proper closed subset $Z\subsetneq X$} if there is an ample line bundle $L$ and a constant $\varepsilon_{X} > 0$ such that every morphism $f: C \to X$ as above satisfying $f(C)\not\subset Z\cup D$ satisfies \eqref{eq:hyperbolicity-definition-log}.
}
\end{definition}

This algebraic property is weaker than its analytic counterpart in the sense that if $X\setminus D$ is Kobayashi hyperbolic and hyperbolically imbedded in $X$, then the log pair $(X,D)$ is log algebraically hyperbolic \cite{PaRo07}.
When $W$ is a surface, pseudo-log algebraic hyperbolicity of $(W,D)$ implies that $W\setminus D$ contains at most finitely many curves of genus 0 or 1 (i.e. elliptic curves, $\pp^1$, $\mathbb{C}$, or $\mathbb{C}^*$).
We say that a surface pair $(W,D)$ of log general type such that $W\setminus D$ contains only finitely many curves of genus 0 or 1 are \textit{pseudo-log Lang algebraically hyperbolic.}
These algebraic properties are what primarily address in this article for various log pairs $(W, D)$.

In terms of examples, these hyperbolicity properties have been most extensively studied and largely settled for subvarieties of abelian varieties, very general hypersurfaces in $\pp^n$, and their complements 
(see e.g. \cite{Kaw80, Siu04, Den16, Bro17, Yam19, BeKi24} for results concerning Kobayashi hyperbolicity).
For algebraic hyperbolicity of very general hypersurfaces $X\subset \pp^n$ of degree $d$ and their complements in $\pp^n$, it is currently known that for $n\ge 3$, a very general hypersurface $X\subset \pp^n$ is both Demailly and Lang algebraically hyperbolic whenever $d\ge 2n-1$ \cite{Ein88, Xu94, Voi96, Voi98, CoRi19}.
This bound can be improved to $d\ge 2n-2$ when $n\ge 6$ \cite{Pac04}.
These properties are expected to hold for very general sextics in $\pp^4$ and octics in $\pp^5$ as well, although currently only the Demailly algebraic hyperbolicity of the octics in $\pp^5$ has been verified \cite{Yeo25}.
For complements, when $n\ge 2$ and $D\subset \pp^n$ is a very general hypersurface, $(\pp^n,D)$ is log algebraically hyperbolic if $d\ge 2n+1$ \cite{Che04, PaRo07}, and pseudo-log algebraically hyperbolic if $d=2n$ \cite{CRY22}.

\subsection{Orbifolds} 
Let $\rho: X' \to X$ be a degree-$m$ cyclic covering branched over a divisor $D$, which endows $X$ with the natural orbifold structure $\left(X,\left(1-\frac{1}{m} \right)D\right)$.
Roulleau and Rousseau \cite{RoRo13} investigated the algebraic hyperbolicity of $X'$ by analyzing orbifold curves mapping to $\left(X,\left(1-\frac{1}{m} \right)D\right)$.
This allows one to use algebraic hyperbolicity results for the associated log pair $(X,D)$ to bound the genus and degree of these curves.

\begin{definition}
{\em
An \emph{orbifold} $(X,\Delta)$ consists of a smooth projective variety $X$ and a Weil divisor $
\Delta = \sum_i \left(1-\frac{1}{m_i}\right)D_i$ where $m_i\in \mathbb{Z}_{>0}\cup \{\infty\}$ and $D_i$ are prime divisors such that $\sum_i D_i$ is a normal crossing divisor. 
}
\end{definition}

\begin{remark}
    This definition is that of a \emph{geometric orbifold} in the sense of Campana \cite{Cam11}. 
    Roulleau and Rousseau \cite{RoRo13} adopted a stricter definition that requires the pair $(X, \Delta)$ to be locally uniformizable.
\end{remark}

When all $m_i=\infty$, the orbifold $(X, \sum_i D_i)$ recovers a log pair. 
When all $m_i=1$, we have $\Delta = \emptyset$, which recovers the projective variety $X$.

Suppose $\rho \colon C' \to C$ is a cyclic covering of degree $m$ between smooth projective curves, branched over a set of points $\{p_i\}$. Assume that for every preimage $q \in \rho^{-1}(p_i)$, the local ramification index is exactly $e_q = m_i$ (which necessarily implies $\operatorname{lcm}_i(m_i) \mid m$). This covering naturally endows $C$ with the structure of an orbifold curve $(C, \sum_i \left(1 - \frac{1}{m_i}\right) p_i)$.
In this setting, the Riemann-Hurwitz formula gives 
\begin{equation*}
    2g(C')-2=m\left(2g(C)-2 + \sum_i\left(1-\frac{1}{m_i}\right)\right).
\end{equation*}
We observe that $C'$ is an elliptic curve when $2g(C)-2 + \sum_i\left(1-\frac{1}{m_i}\right)=0$, and it is a rational curve when $2g(C)-2 + \sum_i\left(1-\frac{1}{m_i}\right)<0$.
This motivates the corresponding definitions for orbifold curves.

\begin{definition}
{\em
An orbifold curve $(C, \sum_i\left(1-\frac{1}{m_i}\right)p_i)$ is:
\begin{itemize}
\item \emph{rational} if $2g(C)-2 + \sum_i\left(1-\frac{1}{m_i}\right)<0$, and
\item \emph{elliptic} if $2g(C)-2 + \sum_i\left(1-\frac{1}{m_i}\right)=0$.
\end{itemize}
}
\end{definition}

\begin{example}[Example 16 in \cite{RoRo13}]\label{example:elliptic-orbicurves}
An orbifold curve $(C, \Delta)$ with $\lfloor \Delta\rfloor = \emptyset$ is elliptic if and only if it is isomorphic to one of the following:
\begin{enumerate}
\item $(E,\emptyset)$, where $E$ is an elliptic curve;
\item $(\pp^1, \left(1-\frac{1}{m_1}\right)p_1+\left(1-\frac{1}{m_2}\right)p_2 + \left(1-\frac{1}{m_3}\right)p_3 )$, where $(m_1,m_2,m_3)=(2,3,6), (2,4,4)$ or $(3,3,3)$;
\item $(\pp^1, \left(1-\frac{1}{2}\right)p_1+\left(1-\frac{1}{2}\right)p_2 + \left(1-\frac{1}{2}\right)p_3+ \left(1-\frac{1}{2}\right)p_4)$.
\end{enumerate}

\end{example}

In the following, we will study curves on surfaces $S'$ that are double covers of a smooth projective surface $S$ branched over a simple normal crossing divisor $D$.
The following lemma provides a precise relationship between the algebraic hyperbolicity of $S'$ and the log algebraic hyperbolicity of $(S,D).$

\begin{lemma}\label{lemma:log-hyp-to-cover-hyp}
{\em
Let $S$ be a smooth projective surface, $D=\sum_jD_j\subset S$ a simple normal crossing divisor, and $\rho:S'\to S$ the cyclic double cover of $S$ branched over $D.$ 
If $2g(C)-2 + \left|f^{-1}(D) \right| \ge \frac{1}{2}D\cdot C$ for all morphisms $f:C\to S$ from smooth projective curves $C$ mapping birationally onto its image such that $f(C)\not\subset D$, then any rational curve in $S'$ is contained in the ramification divisor of $\rho.$
}
\end{lemma}

\begin{proof}
Let $f'\colon C' \to S'$ be a morphism from a smooth projective curve $C'$ that is birational onto its image and not contained in the ramification divisor.
Let $C$ denote the desingularization of the curve $\rho\circ f'(C')\subset S$ and let $f\colon C \to S$ denote the induced morphism, so $f(C)\not\subset D$.
The map $\rho\circ f'$ factors through $f$, giving a branched cover $\pi\colon C' \to C$ of degree 1 or 2. 
\begin{equation*}
\begin{tikzcd}
C' \arrow[r, "f'"] \arrow[d, "\pi"'] & S' \arrow[d, "\rho"] \\
C \arrow[r, "f"']                     & S
\end{tikzcd}
\end{equation*}

Let $f^*D = \sum_i t_i\cdot p_i$, so the branched cover $\pi \colon C' \to C$ induces a natural orbifold structure $(C, \sum_i\left(1-\frac{1}{m_i}\right)p_i)$ which satisfies $2 \le m_i \cdot t_{i}.$
Note that we have $\sum_i \left(1-\frac{1}{m_i}\right)\geq \left|f^{-1}(D) \right|-\frac{1}{2}D\cdot C$. This implies
$$2g(C)-2+ \sum_i \left(1-\frac{1}{m_i}\right) \ge 2g(C)-2+|f^{-1}(D)|-\frac{1}{2}D\cdot C,$$
which is non-negative by assumption.
Therefore, by the Riemann-Hurwitz formula applied to $\pi$,$$ 2g(C')-2 \ge 2g(C)-2+ \sum_i \left(1-\frac{1}{m_i}\right) \ge 0. $$

\end{proof}

\begin{remark}\label{rem:elliptic-curves-double cover}
    In the setting of Lemma~\ref{lemma:log-hyp-to-cover-hyp}, one notices that if $C'$ is an elliptic curve in $S'$ not contained in the ramification divisor of $\rho$, then this can only occur when $2g(C)-2+|f^{-1}(D)|=\frac{1}{2}D\cdot C$ and $m_i\cdot t_i=2$ for all $i.$
    Thus, $C$ is either:
    \begin{enumerate}
        \item an elliptic curve, and $m_i = 1$ (with $t_i = 2$) for all $i$; or

        \item a rational curve, and there are exactly four $i\in I$ such that $m_i = 2$ (with $t_i = 1$), and $m_i = 1$ (with $t_i = 2$) for all other $i$.
    \end{enumerate}
\end{remark}

\begin{remark}
    In the setting of Lemma~\ref{lemma:log-hyp-to-cover-hyp}, assume that no component $D_j$ of $D$ is of geometric genus 0 or 1.
    If there is an ample line bundle $L$ on $S$ and an $\varepsilon>0$ such that $2g(C)-2 + \left|f^{-1}(D) \right| \ge \frac{1}{2}D\cdot C+\varepsilon L\cdot C$ for all such curves on $S$ as above, then $S'$ is Demailly algebraically hyperbolic. 
    Moreover, if $D$ is ample and $(S,D)$ is log algebraically hyperbolic, specifically with
\begin{equation*}
    2g(C)-2+|f^{-1}(D)| \ge \left(\frac{1}{2}+\varepsilon\right) D\cdot C
\end{equation*}
for some $\varepsilon>0$, then $S'$ is Demailly algebraically hyperbolic.

\end{remark}

\subsection{Logarithmic sheaves}\label{sec:log-sheaves}

Let $(X,D)$ be a log smooth pair, i.e. $X$ is a smooth projective variety and $D$ a simple normal crossing divisor. 
The \emph{log tangent sheaf} $T_{X}(-\log D)$ is the kernel sheaf of the natural map $T_{X}\to \oo_D(D)$: 
\begin{equation*}
0 \to T_{X}(-\log D) \to T_{X}\to \oo_D(D)\to \oo_{D_{\text{sing}}}(D) \to 0.
\end{equation*} 
We denote by $N_{D/X}'$ the equisingular normal sheaf of $D$ in $X$, which is the kernel of $ \oo_D(D)\to \oo_{D_{\text{sing}}}(D).$
For example, when $D$ is a simple normal crossing divisor with components $D_j$, $N'_{D/X}\simeq \oplus_j\oo_{D_j}(D_j).$ 
For a log smooth pair $(X,D)$ with $\dim X=n$, $T_{X}(-\log D)$ is a vector bundle on $X$ of rank $n.$
We refer the reader to \cite{Ser06} for a more comprehensive treatment on logarithmic sheaves.

Let $f:C\to X$ be a non-constant map from a smooth projective curve $C$ such that $f(C)\not\subset D$, and consider the reduced scheme of points $D_f=f^{-1}(D)$ on $C.$
As above, we may define the \textit{log tangent sheaf} $T_C(-\log D_f)$ as the kernel bundle in the following short exact sequence:
\begin{equation*}
0 \to T_{C}(-\log D_f) \to T_{C}\to \oo_{D_f}(D_f)\to 0.
\end{equation*}

We define the \emph{log normal sheaf} $N_{f/X}(\log D)$ as the quotient sheaf of the inclusion $T_C(-\log D_f)\to f^*T_{X}(-\log D)$:
\begin{equation*}
0 \to T_C(-\log D_f)\to f^*T_{X}(-\log D)\to N_{f/X}(\log D)\to 0.
\end{equation*}
It is a coherent sheaf on $C$ of rank $n-1$ and degree equal to
\begin{equation}\label{eq:deg-log-normal}
\deg N_{f/X}(\log D) = 2g(C) - 2 + |f^{-1}(D)|-\deg f^*(K_X+D).
\end{equation}

\subsection{A variational argument}\label{sec:variational-argument}

Let $X$ be a smooth projective variety and $E$ an effective divisor on $X$. Suppose the linear system decomposes as $|E| = |E'| + F$, where the fixed divisor $F$ has simple normal crossing support and the moving part $|E'|$ is basepoint-free.

Let $B_1=H^0(X,E)$, and let $\D_1\to B_1$ be the universal hypersurface.
Fix some positive integers $g$ and $i$, and an effective curve class $\beta \in H_2(X, \mathbb{Z})$. 
We denote by $\overline{\mathcal{M}}_{g,i}(X, \beta)$ the Kontsevich moduli stack parametrizing stable maps $f \colon C \to X$ from an $i$-pointed nodal curve $C$ of arithmetic genus $g$ satisfying $f_*[C] = \beta$. 
We consider the open subscheme $\mathcal{K} \subset \overline{\mathcal{M}}_{g,i}(X, \beta)$ parametrizing maps where $f:C\to X$ is a \textit{birational morphism onto its image} and $C$ is a \textit{smooth} curve. 

Denote by $\Sigma\subset \mathcal{K}\times B_1$ the incidence subscheme parametrizing $((f:C\to X, p_1, \ldots, p_i), D)$ such that the support of $f^*D$ is $\{p_1, \ldots, p_i\}.$
Suppose that for a general hypersurface $D\in B_1$, there exists a map $f:C\to X$ from a smooth projective curve $C$ of genus $g$ that is birational onto its image, satisfies $f(C)\not\subset D$ and $f_*[C]=\beta$, and meets $D$ in exactly $i$ distinct points.
In this case, the projection map $\psi: \Sigma \to B_1$ would be dominant.
We let $\cc_1\rightarrow \Sigma$ be the universal curve over $\Sigma$.

At a general point $\sigma\in \Sigma,$ we may choose a locally closed subvariety $\Sigma_0\subset \Sigma$ through $\sigma$ such that $\psi\vert_{\Sigma_0}:\Sigma_0\to B_1$ is \'etale at $\sigma$.
We denote the base change of $\D_1\to B_1$ over $\Sigma_0$ by $\mathcal{D}_2 \to \Sigma_0$.
Finally, shrink $\Sigma_0$ to an open subset $B\subset \Sigma_0$ so that the base change of $\D_2\to\Sigma_0$ over $B$ has simple normal crossing fibers. 

We establish the following notation for the families restricted over $B$:
\begin{enumerate}
    \item $\mathcal{X} = X \times B$ is the trivial ambient family.
    \item $\mathcal{C} \to B$ denotes the base change of $\mathcal{C}_1 \to \Sigma$ over $B$, and $f: \mathcal{C} \to \mathcal{X}$ denotes the natural map, which is birational onto its image.
    \item $\mathcal{D} \to B$ denotes the base change of the universal hypersurface over $B_1$ to $B$, and $g: \mathcal{D} \hookrightarrow \mathcal{X}$ denotes the natural closed immersion.
    \item $\mathcal{D}' \subset \mathcal{D}$ denotes the subfamily over $B$ consisting of the irreducible components of $\mathcal{D}$ that correspond to the basepoint-free locus.
    $\mathcal{F} \subset \mathcal{D}$ denotes the subfamily consisting of the fixed component of $\mathcal{D}$, i.e. $\mathcal{F}\simeq F\times B.$
    \item We denote the reduced pullback divisor by $\mathcal{D}_f = (f^{-1}(\mathcal{D}))_{red}$, and its natural closed immersion by $h: \mathcal{D}_f \hookrightarrow \mathcal{C}$.
    \item $\pi_1: \mathcal{X} \to B$ and $\pi_2: \mathcal{X} \to X$ are the natural projection maps.
    \item For a closed point $b \in B$, we denote the fibers of $\mathcal{C}$, $\mathcal{D}$, $\mathcal{D}'$, and $\mathcal{D}_f$ over $b$ by $C_b$, $D_b$, $G_b$, and $D_{f,b}$ respectively, and the restriction of $f$ by $f_b: C_b \to X$.
\end{enumerate}

\begin{center}
\begin{tikzcd}[row sep=large, column sep=large]
\mathcal{D}_f \arrow[r, hook, "h"] \arrow[dr] & \mathcal{C} \arrow[r, "f"] \arrow[d] & \mathcal{X} \arrow[r, "\pi_2"] \arrow[dl, "\pi_1"] & X \\
& B & \mathcal{D} \arrow[u, hook, "g"] \arrow[l] & \\
& & \mathcal{D}' \arrow[u, hook] &
\end{tikzcd}
\end{center}

\begin{remark}
    We will focus on the case where $X$ is a regular surface, denoted $S$, meaning $H^1(S, \mathcal{O}_S) = 0$.
This ensures that if $\beta$ is a linear equivalence class in $\operatorname{Pic}(S)$, then the curves parametrized by $\overline{\mathcal{M}}_{g,i}(S, \beta)$ are all linearly equivalent to $\beta$.
\end{remark}

Since $(\mathcal{X}, \D)$, $(\mathcal{X}, \mathcal{D}')$ and $(\cc, \mathcal{D}_f)$ are log smooth pairs, we may define the appropriate log tangent and normal sheaves as in \S\ref{sec:log-sheaves}.
These sheaves are defined by the following short exact sequences:
\begin{equation}\label{eq:log-tgt-D-in-X-family}
0 \to T_{\mathcal{X}}(-\log \mathcal{D}) \to T_\mathcal{X} \to N_{\D/\mathcal{X}}'\to 0
\end{equation}
\begin{equation}\label{eq:log-tgt-D'-in-X-family}
0 \to T_{\mathcal{X}}(-\log \mathcal{D}') \to T_\mathcal{X} \to \oo_{\mathcal{D}'}(\mathcal{D}')\to 0
\end{equation}
\begin{equation*}
0 \to T_{\mathcal{C}}(-\log \mathcal{D}_f) \to T_\mathcal{C} \to N'_{\mathcal{D}_f/\mathcal{C}}\to 0
\end{equation*}
\begin{equation}
0 \to T_{\mathcal{C}}(-\log \mathcal{D}_f) \to f^*T_\mathcal{X}(-\log \mathcal{D}) \to N_{f/\mathcal{X}}(\log \mathcal{D})\to 0
\end{equation}

Note that $\D=\mathcal{F}\cup \mathcal{D}'$, and $N_{\D/\mathcal{X}}'\simeq \oo_{\mathcal{D}'}(\mathcal{D}')\oplus N'_{\mathcal{F}/\mathcal{X}}.$
We define $T_{\mathcal{X}/X}(-\log \mathcal{D}')$ as the relative log tangent sheaf of the projection $\pi_2\colon \mathcal{X}\to X$:

\begin{equation}\label{eq:rel-log-tgt-sheaf-family}
0 \to T_{\mathcal{X}/X}(-\log \mathcal{D}') \to T_\mathcal{X}(-\log \mathcal{D}') \to \pi_2^*T_{X}\to 0.
\end{equation}
This sequence is right exact because $E'$ is basepoint-free.

The following result and proof are adapted from \cite[Lemmas 2.2-2.3]{CoRi19} and \cite[Propositions~3.3-3.5]{CRY22}.

\begin{proposition}\label{prop:normal-family}
{\em
Assume the setup as presented above from the beginning of \S\ref{sec:variational-argument}.
For a general $b\in B$,
\begin{enumerate}
\item  $N_{f/\mathcal{X}}(\log \mathcal{D})\vert_{C_b}\simeq N_{f_b/X}(\log D_{f,b})$.
\item $N_{f_b/X}(\log D_{f,b})$ is the extension of two sheaves $Q_1$ and $Q_2$ on $C_b$:
\begin{equation}
    0 \to Q_1 \to N_{f_b/X}(\log D_{f,b}) \to Q_2 \to 0
\end{equation}
where $Q_1$ is a quotient of $f_b^*T_{\mathcal{X}/X}(-\log \mathcal{D}')$ and $Q_2$ is a quotient of $f_b^*\pi_2^*T_X(-\log F).$
\end{enumerate}
}
\end{proposition}

\begin{proof}

(1) holds because $C_b$ and $X_b\simeq X$ are fibers of fibrations $\mathcal{C}\to B$ and $\pi_1\colon \mathcal{X}\to B.$

$T_\mathcal{X}$ splits naturally as $\pi_1^*T_B\oplus \pi_2^*T_X$, where $\pi_1^*T_B\simeq H^0(X, E') \otimes \oo_\mathcal{X}.$ 
By a standard diagram chase involving \eqref{eq:log-tgt-D'-in-X-family} and \eqref{eq:rel-log-tgt-sheaf-family} (see the proof of Proposition 3.7 in \cite{CRY22}), we have
\begin{equation*}
0 \to T_{\mathcal{X}/X}(-\log \mathcal{D}') \to H^0(X, E') \otimes \oo_\mathcal{X} \to \mathcal{O}_{\mathcal{D}'}(\mathcal{D}') \to 0.
\end{equation*}
This, together with \eqref{eq:log-tgt-D-in-X-family} and the splitting of $N'_{\D/\mathcal{X}}$, induces the following commutative diagram:

\begin{center}
\begin{tikzcd}
& 0 \arrow[d] & 0 \arrow[d] & 0 \arrow[d] & \\
0 \arrow[r] & T_{\mathcal{X}/X}(-\log \mathcal{D}') \arrow[r] \arrow[d] & H^0(X, E') \otimes \oo_\mathcal{X}\arrow[r] \arrow[d] & \mathcal{O}_{\mathcal{D}'}(\mathcal{D}') \arrow[r] \arrow[d] & 0 \\
0 \arrow[r] & T_{\mathcal{X}}(-\log \mathcal{D}) \arrow[r] \arrow[d] & T_{\mathcal{X}} \arrow[r] \arrow[d] & N'_{\mathcal{D}/\mathcal{X}} \arrow[r] \arrow[d] & 0 \\
0 \arrow[r] & \pi_2^*T_X(-\log F) \arrow[r] \arrow[d] & \pi_2^*T_X \arrow[r] \arrow[d] & N'_{\mathcal{F}/\mathcal{X}}\arrow[r] \arrow[d] & 0 \\
& 0 & 0 & 0 &
\end{tikzcd}
\end{center}
Note that in the bottom row, the kernel of $\pi_2^*T_X\to N'_{\mathcal{F}/\mathcal{X}}$ is $\pi_2^*T_X(-\log F)$ because $\mathcal{F}=F\times B.$ 
Hence by the snake lemma, it is the cokernel in the leftmost column.

Since $f^*T_\mathcal{X}(-\log \D)$ surjects onto $N_{f/\mathcal{X}}(\log \mathcal{D})$, the leftmost column gives the following exact commuting diagram, where $\mathcal{Q}_1,\mathcal{Q}_2$ are some sheaves on $\mathcal{C}$:
\begin{center}
\begin{tikzcd}
0 \arrow[r] & f^*T_{\mathcal{X}/X}(-\log \mathcal{D}') \arrow[r] \arrow[d] & f^*T_{\mathcal{X}}(-\log \mathcal{D}) \arrow[r] \arrow[d] & f^*\pi_2^*T_X(-\log F) \arrow[r] \arrow[d] & 0 \\
0 \arrow[r] & \mathcal{Q}_1 \arrow[r] \arrow[d] & N_{f/\mathcal{X}}(\log \mathcal{D})\arrow[r] \arrow[d] & \mathcal{Q}_2 \arrow[r] \arrow[d] & 0 \\
& 0 & 0 & 0 &
\end{tikzcd}
\end{center}
(2) then follows from restricting this diagram to $C_b$ and applying (1).
(Note that $Q_1$ may not be isomorphic to the restriction $\mathcal{Q}_1\vert_{C_b}$, but it is a quotient of $\mathcal{Q}_1\vert_{C_b}$.)

\end{proof}

\subsection{Kernel bundles}
Now we explain how kernel bundles enter the picture.
For any globally generated vector bundle $\mathcal{E}$ on $X$, we define its \emph{kernel bundle} $M_\mathcal{E}$ as the kernel of the section evaluation map $\operatorname{ev}: H^0(X,\mathcal{E})\otimes \oo_{X}\to \mathcal{E}$.
\begin{equation*}
0 \to M_\mathcal{E}\to H^0(X,\mathcal{E})\otimes \oo_{X}\xrightarrow{\operatorname{ev}} \mathcal{E}\to 0
\end{equation*}
The following definition of a section-dominating collection $L_i$ of line bundles for a globally generated vector bundle $\mathcal{E}$ was introduced in \cite{CoRi23}.

\begin{definition}
\label{def:sec-dom}
{\em
Let $\mathcal{E}$ be a globally generated vector bundle on a smooth projective variety $X$. 
We denote the ideal sheaf of a point $p\in X$ by $\mathcal{I}_p.$
Let $L_1,\ldots, L_u$ be a collection of non-trivial globally generated line bundles such that $\mathcal{E}\otimes L_i^\vee$ is globally generated for each $i=1,\ldots,u$.
\begin{itemize}
\item We say that $L_1,\ldots, L_u$ \emph{dominates the sections of $\mathcal{E}$} if the map 
\begin{equation*}
\bigoplus_{i=1}^u\left(H^0(L_i\otimes \mathcal{I}_p)\otimes H^0(\mathcal{E}\otimes L_i^\vee)\right)\to H^0(\mathcal{E}\otimes \mathcal{I}_p)
\end{equation*}
is surjective for every point $p\in X$.
\item Let $Z\subsetneq X$ be a proper closed subset. We say that $L_1,\ldots, L_u$ \emph{dominates the sections of $\mathcal{E}$ outside of $Z$} if the map 
\begin{equation*}
\bigoplus_{i=1}^u\left(H^0(L_i\otimes \mathcal{I}_p)\otimes H^0(\mathcal{E}\otimes L_i^\vee)\right)\to H^0(\mathcal{E}\otimes \mathcal{I}_p)
\end{equation*}
is surjective for every point $p\in X\setminus Z$.
\end{itemize}
}
\end{definition}

\begin{remark}
If $\mathcal{E}$ is a vector bundle whose sections are dominated by $L_1,\ldots,L_u$, then there is a surjection 
\begin{equation*}
    M_{L_1}\ ^{\oplus s_1} \oplus \cdots \oplus M_{L_u}\ ^{\oplus s_u} \to M_\mathcal{E}
\end{equation*}
for some $s_1,\ldots,s_u>0$.
Since $M_{L_i}\otimes L_i$ is globally generated for each $i=1,\ldots,u$, this implies that $M_\mathcal{E}\otimes L_1\otimes \cdots \otimes L_u$ is globally generated
(see \cite{HMP10} for related discussion).
\end{remark}

In the cases of effective divisors $E=E'+ F$ (where $F$ is fixed and $E'$ is basepoint-free) that we will consider, $T_X(-\log F)$ is 
pseudo-nef outside a proper closed subset $Z\subsetneq X$ (see \cite[Definition 2.6]{IMRY25}).
In this setting, for a general hypersurface $D\in H^0(X,E)$ and a map $f:C\to X$ from a smooth projective curve $C$ that is birational onto its image and satisfies $f(C)\not\subset D\cup Z$, the following proposition shows that the degree $\deg N_{f/X}(\log D)$ of its log normal sheaf is bounded from below by the degree of the restricted kernel bundle $f^*M_{E'}$, which in turn is bounded in terms of the degrees of the restricted section-dominating line bundles $f^*{L_i}$.

\begin{proposition}\label{prop:log-normal-degree-bound}
{\em
With the setup and notation of \S\ref{sec:variational-argument}, $\pi_2^*M_{E'}$ injects into $T_{\mathcal{X}/X}(-\log \mathcal{D}')$ with trivial cokernel $\oo_\mathcal{X}.$
Assume further that $T_X(-\log F)$ is pseudo-nef outside of a proper closed subset $Z\subsetneq X$, and that for a general hypersurface $D\in H^0(X,E)$, there exists a map $f:C\to X$ from a smooth projective curve $C$ of genus $g$ that is birational onto its image, satisfies $f(C)\not\subset D\cup Z$ and $f_*[C]=\beta$, and meets $D$ in exactly $i$ distinct points.
Then for a general $b\in B$,
\begin{equation}\label{eq:deg-log-normal-bound}
    \deg N_{f_b/X}(\log D_{f,b})\ge \deg f_b^*M_{E'}.
\end{equation}
Moreover, if $X$ is a surface (denoted $S$) and $\deg N_{f_b/S}(\log D_{f,b})<0$, then there is a generically surjective map 
\begin{equation}\label{eq:log-normal-gen-surj-map}
    f_b^*M_{E'}\to N_{f_b/S}(\log D_{f,b}).
\end{equation}
}
\end{proposition}

\begin{proof}
The first statement that $\pi_2^*M_{E'}$ injects into $T_{\mathcal{X}/X}(-\log \mathcal{D}')$ with trivial cokernel follows from the same proof as in \cite[Proposition 3.7]{CRY22}.
Let $Q_1$ and $Q_2$ be the sheaves on $C_b$ from Proposition~\ref{prop:normal-family}.
Let $K_1$ and $K_2$ be sheaves on $C_b$ completing the diagram below:

\begin{center}
\begin{tikzcd}
0 \arrow[r] & f_b^*M_{E'} \arrow[r] \arrow[d] & f_b^*T_{\mathcal{X}/X}(-\log \mathcal{D}') \arrow[r] \arrow[d] & \oo_{C_b} \arrow[r] \arrow[d] & 0 \\
0 \arrow[r] & K_1 \arrow[r] \arrow[d] & Q_1\arrow[r] \arrow[d] & K_2 \arrow[r] \arrow[d] & 0 \\
& 0 & 0 & 0 &
\end{tikzcd}
\end{center}
Since $M_{E'}$ is a subbundle of a trivial bundle, $\deg K_1\ge \deg f_b^*M_{E'}$, which implies that $\deg Q_1\ge\deg M_{E'}\vert_{C_b}.$
Since $T_X(-\log F)$ is pseudo-nef outside of $Z$ and the curve $C_b$ does not map into $Z$, $Q_2$ has non-negative degree. 
The required inequality \eqref{eq:deg-log-normal-bound} then follows.
Moreover, suppose $X=S$ is a surface, which implies $N_{f_b/S}(\log D_{f,b})$ has rank one. 
If $\deg N_{f_b/S}(\log D_{f,b})<0$, then $K_1$ must have rank one.
This means that the natural map \eqref{eq:log-normal-gen-surj-map} is generically surjective.
\end{proof}

\begin{remark}\label{rem:log-normal-degree-bound}
Suppose we are in the setting of Proposition~\ref{prop:log-normal-degree-bound}.
Assume $X=S$ is a surface and that $L_i$ are globally generated line bundles dominating the sections of $E'$ outside of  $Z\subsetneq X$.
For a general $b\in B$, if $\deg N_{f_b/S}(\log D_{f,b})<0$, then there is a generically surjective map 
\begin{equation*}
    M_{L_i}\vert_{C_b}\to N_{f_b/S}(\log D_{f,b})
\end{equation*}
for some $i$.
This implies that $\deg N_{f_b/S}(\log D_{f,b})\ge \min\limits_{i} \{\deg f_b^*M_{L_i}\}.$
Since the $L_i$ tend to be less positive than $E'$, this gives a stronger (less negative) lower bound on $\deg N_{f_b/S}(\log D_{f,b}).$
\end{remark}

We will need the following elementary result about the sections of the pullback of kernel bundles to curves.

\begin{proposition}\label{prop:sections-of-kernel-bundles}
{\em
Let $L$ be a globally generated line bundle on a smooth projective variety $X$, and let $f:C\to X$ be a non-constant map from a smooth projective curve $C.$ Then
$$H^0(C, f^*M_L)\simeq H^0(X, \mathcal{I}_{f(C)/X}\otimes L).$$
}
\end{proposition}

\begin{proof}
By applying the global section functor to the short exact sequences 
\begin{center}
 $0\to f^*M_L\to  \mathcal{O}_C\otimes H^0(X,L)\to f^*L\to 0$ \quad  and \quad $0\to \mathcal{I}_{f(C)/X} \otimes  L\to L\to  L|_{f(C)}\to 0 $, 
\end{center}
we obtain the following exact sequences: 
\begin{equation*}
    \xymatrix{
    0 \ar[r] & H^0(C,f^*M_L)  \ar[r] &  H^0(X,L)\ar[r] & H^0(C,f^*L) \\
    0 \ar[r] & H^0(X,\mathcal{I}_{f(C)/X} \otimes  L)\ar[r] & H^0(X,L)\ar[r]  &   H^0(f(C),L|_{f(C)})
    }   
\end{equation*}
The map $H^0(X,L) \to H^0(C,f^*L)$ sends a section $s \in H^0(X,L)$ to its pullback $f^*s \in H^0(C,f^*L)$. 
Thus, $s$ lies in the kernel $H^0(C,f^*M_L)$ if and only if $f^*s = 0$.
Similarly, the map $H^0(X,L) \to H^0(f(C),L|_{f(C)})$ sends a section $s$ to its restriction $s|_{f(C)}$. 
Since $f$ is non-constant, a section vanishes on the image $f(C)$ if and only if its pullback to $C$ vanishes, i.e. $f^*s = 0$ if and only if $s|_{f(C)} = 0$. 
\end{proof}

\section{Hyperbolicity of very general Horikawa surfaces}\label{sec:horikawa-hyp}

In this final section, we apply the tools developed in the previous section to study the hyperbolicity of very general Horikawa surfaces. We do so by analyzing double covers of a base surface $W$, where $W$ is one of the following:
\begin{itemize}
\item the projective plane $\mathbb{P}^2$ (\S\ref{sec:p2});
\item a Hirzebruch surface $\mathbb{F}_d$ (\S\ref{sec:hirzebruch});
\item the blowup of a Hirzebruch surface $\mathbb{F}_d$ at two distinct points lying on the same fiber $\Gamma_0\in |\Gamma|$ (\S\S\ref{sec:hirzebruch-2pt}-\ref{sec:hirzebruch-2pt-3comp}); or
\item the blowup of a Hirzebruch surface $\mathbb{F}_d$ with $d>0$ at a point $x\in \Delta_0$ (\S\ref{sec:hirzebruch-1pt}).
\end{itemize}
In all cases, the double cover is branched over a divisor $D$ that is very general in its linear system.
The methods are the same across each case. 
If the branch divisor decomposes as $D=D'+F$, where $D'$ is the basepoint-free part and $F$ is the fixed part, the strategy is as follows: 
We construct a section-dominating collection (Definition~\ref{def:sec-dom}) of divisors for $D'$ outside a proper closed set $\Sigma_1\subset W$, and demonstrate that the log tangent bundle $T_W(-\log F)$ is pseudo-nef outside a proper closed set $\Sigma_2\subset W$.
With these conditions met, Proposition~\ref{prop:log-normal-degree-bound} yields log algebraic hyperbolicity bounds for the pair $(W,D)$, valid for all curves not contained in $\Sigma_1\cup \Sigma_2$. 
We then utilize Proposition~\ref{prop:sections-of-kernel-bundles} to explicitly characterize the exceptional curves on $W$ that achieve strict equality in these bounds.
Finally, via Lemma~\ref{lemma:log-hyp-to-cover-hyp}, we lift the hyperbolicity of the log pair $(W,D)$ to the double cover $S'$ branched over $D$. 
Consequently, any geometrically rational or elliptic curve on $S'$ must arise either as the pullback of a curve in the exceptional locus $\Sigma_1\cup \Sigma_2$, or as the pullback of a curve on $W$ achieving equality in the degree bounds.

\subsection{The projective plane}\label{sec:p2}
Horikawa surfaces of the first kind of type $(\infty)$ consist of two cases: double covers of $\mathbb{P}^2$ branched over a curve $D$ of degree 8 ($n=2$) or degree 10 ($n=5$).
As mentioned earlier, \cite{RoRo13} established that double covers branched over a very general $D \in |10H|$ are Demailly algebraically hyperbolic.
Concerning double covers branched over a very general $D \in |8H|$, Roulleau and Rousseau proved that these do not contain rational curves, but they are known to contain elliptic curves arising from the preimages of bitangent lines to $D$.
We will show that these are the only elliptic curves on the double cover. 
Consequently, these surfaces are pseudo-Lang algebraically hyperbolic.
Whether they are pseudo-Demailly algebraically hyperbolic, however, remains an open question.

It is a standard fact that $T_{\mathbb{P}^2}$ is a nef vector bundle and that $\oo_{\pp^2}(1)$ dominates the sections of $\mathcal{O}_{\mathbb{P}^2}(d)$ for all $d \ge 1$.
We denote the kernel bundle of $\mathcal{O}_{\mathbb{P}^2}(d)$ by $M_d$.

\begin{lemma}\label{lemma:secdom-p2}
{\em
Let $d\ge 1$. At every point $p\in \pp^2$, the natural multiplication map
\begin{equation*}
    H^0(\pp^2, \oo_{\pp^2}(1)\otimes \mathcal{I}_{p/\pp^2})\otimes H^0(\pp^2, \oo_{\pp^2}(d-1))\to H^0(\pp^2, \oo_{\pp^2}(d)\otimes \mathcal{I}_{p/\pp^2}),
\end{equation*}
is surjective. Consequently, there is a surjective map
\begin{equation*}
M_1^{\oplus s}\to M_d
\end{equation*}
where $s=h^0(\pp^2,\oo_{\pp^2}(d-1))$.
}
\end{lemma} 

Chen proved the following result, which implies that for a very general curve $D\subset \pp^2$ of degree $d\ge5$, the pair $(\pp^2,D)$ is log algebraically hyperbolic. 
When $D$ is a very general quartic, $(\pp^2, D)$ is not strictly log algebraically hyperbolic due to the presence of bitangent and flex lines to $D$. 
It was proved in \cite{CRY22} that $(\pp^2,D)$ is pseudo-log algebraically hyperbolic outside of the union of these lines; hence, it is pseudo-log Lang algebraically hyperbolic. 

\begin{theorem}[{\cite[Theorem 1.7]{Che01}}]\label{theorem:xi-chen}
{\em
Let $D\subset \pp^2$ be a very general curve of degree $d.$ 
Then 
\begin{equation*}
2g(C)-2+\left|\nu^{-1}(D) \right| \geq (d-4)\deg (C)
\end{equation*}
for all integral curves $C\neq D$, where $g(C)$ denotes the geometric genus of $C$ and $\nu:\widetilde{C}\to \pp^2$ denotes the induced map from the normalization of $C.$
}
\end{theorem}

We prove that in Theorem~\ref{theorem:xi-chen}, equality is strictly achieved only when $C$ is a line.
Note that this immediately establishes that when $d=4$ and $D$ is a very general quartic, $(\pp^2,D)$ is pseudo-log Lang algebraically hyperbolic.

\begin{proposition}\label{prop:p2-horikawa}
{\em
Let $D\subset \pp^2$ be a very general curve of degree $d$, and let $f:C\to \pp^2$ be a non-constant map from a smooth projective curve $C$ of genus $g$ that is birational onto its image such that $\deg f(C)=e$, $f(C)\not\subset D$ and $|f^{-1}(D)|=i.$
Then $2g-2+i\ge(d-4)e$, and equality is achieved only when $C$ is a line.
}
\end{proposition}

\begin{proof}
We will use the setup and notation from \S\ref{sec:variational-argument} in this setting.
Recall from \eqref{eq:deg-log-normal} that at a general $b\in B$:
\[\deg N_{f_b/\pp^2}(\log D_{f,b}
)=2g-2+i-(d-3)e, \]
so we may assume that $\deg N_{f_b/\pp^2}(\log D_{f,b}
)<0.$
Since $T_{\pp^2}$ is nef, Proposition~\ref{prop:log-normal-degree-bound} and Lemma~\ref{lemma:secdom-p2} imply that there is a generically surjective map
\[ \mu:f_b^*M_1 \to N_{f_b/\pp^2}(\log D_{f,b}). \]
Let us denote the 
image sheaf of $\mu$ by $Q$.
Since $M_1$ is a subbundle of a trivial bundle, we have $\deg Q\ge \deg f_b^*M_1 = -e$, and so 
\[2g-2+i\ge \deg Q + (d-3)e\ge (d-4)e. \]

Now assume that $2g-2+i=(d-4)e$, which forces $\deg Q =-e$ and the kernel $K$ of $\mu$ to have $\deg K=0.$
Since $K$ is a rank-one subsheaf of $\oo_{C_b}^{\oplus 3}$, projecting onto at least one of the trivial factors yields a non-zero map $\nu \colon K\to \oo_{C_b}$.
A non-zero morphism between line bundles of the same degree must be an isomorphism (Schur's Lemma). 
Hence, $\oo_{C_b}$ injects into $f_b^*M_1$, and $f_b^*M_1$ must have global sections.
By Proposition~\ref{prop:sections-of-kernel-bundles}, we have 
$$h^0(C_b, f_b^*M_1) = h^0(\pp^2, \mathcal{I}_{f_b(C_b)/\pp^2}(1)) = h^0(\pp^2, \oo_{\pp^2}(1-e)),$$
which strictly forces $e=1$.
\end{proof}

Combined with Lemma \ref{lemma:log-hyp-to-cover-hyp} and Remark~\ref{rem:elliptic-curves-double cover}, we obtain the following result.

\begin{corollary}\label{cor:p2-horikawa-cover}
{\em
Let $\rho:S'\to \pp^2$ be a double cover branched over a very general degree 8 plane curve $D.$ Then:
\begin{enumerate}
    \item $S'$ does not contain rational curves.
    \item If $C'\subset S'$ is an integral curve with geometric genus $g(C')=1$, then $C'$ is the pullback of a line $\ell\subset \pp^2$ that is bitangent to $D.$
In particular, since there are only finitely many bitangent lines to $D$, there are only finitely many geometrically elliptic curves on $S'.$
\end{enumerate}
}
\end{corollary}

By the classical Pl\"ucker formulas, the number of bitangent lines to a smooth plane curve of degree $d$ is 
\[ \frac{1}{2}d^4-d^3-\frac{9}{2}d^2+9d. \]
In particular, when $d=8$, this number is 1320.

Thus, we have shown that a very general Horikawa surface of the first kind with $n=2$ (which is of type $(\infty)$) is pseudo-Lang algebraically hyperbolic. 
Specifically, it contains no rational curves and exactly 1320 geometrically elliptic curves.

\subsection{Hirzebruch surfaces}\label{sec:hirzebruch}
Let $\ff_d$ be the $d$-th Hirzebruch surface, which comes with a natural projection $\ff_d\to \pp^1$. We denote by $\Gamma$ the class of a fiber of this projection and by $\Delta_0$ the class of the zero-section 
($\Gamma^2 = 0$, $\Delta_0 \cdot \Gamma = 1$ and $\Delta_0^2=-d$). 
When $d=0$, we maintain the same notation by letting $\Delta_0$ and $\Gamma$ denote the fiber classes of the two natural projections.
When $d>0$, the zero section $\Delta_0$ is a fixed curve.

Besides the type $(\infty)$ surfaces, all other Horikawa surfaces of the first kind are of type $(d)$ or $(d')$ and are therefore birational to double covers of Hirzebruch surfaces. 
Additionally, Horikawa surfaces of the second kind of type $(2^*)$ are also realized as double covers of Hirzebruch surfaces. 

For each geometric genus determined by the invariant $n$, the moduli of Horikawa surfaces of the first kind is stratified by types (see Table~\ref{table:horikawa-moduli}).
For a general point in each stratum, the branch locus is either a general divisor, or the union of a general divisor with the fixed curve $\Delta_0$ (when $d>0$).

Roulleau and Rousseau \cite{RoRo13} previously established that for $n\ge 3$ and type $(\frac{n-3}{3})$, a very general surface contains no rational curves. 
It does, however, contain elliptic curves arising as pullbacks of curves of class $\Gamma$ (and also $\Delta_0$ when $n=3$) that are tangent to the branch locus $D$. 
We prove that these are, in fact, the only elliptic curves on such a surface.
In fact, we will establish pseudo-log Lang algebraic hyperbolicity for a very general point in every stratum of moduli for Horikawa surfaces of the first kind. 
Specifically, we will show that the only possible rational curves are the exceptional curves arising from the nodes of $D$ and the curves contained in the pullback of $\Delta_0$, and the only possible elliptic curves are those pulled back from either $\Delta_0$ or from curves of class $\Gamma$ (or possibly $\Delta_0$ as well when $d=0$ and $n=3$) that are tangent to $D$ or intersect $D$ at a node.

Recall that the canonical class of $\ff_d$ is given by:
$$K_{\ff_d} = -2\Delta_0 - (d+2)\Gamma.$$
We also note that the tangent bundle $T_{\ff_0}$ is nef, the logarithmic tangent bundle $T_{\ff_d}(-\log \Delta_0)$ is nef for $d>0$, and the tangent bundle $T_{\ff_d}$ is pseudo-nef outside of $\Delta_0$ when $d>0$.

The following section-dominance result was proved in \cite[Example 2.6]{CoRi23}.
It uses the vanishing result $H^1(\mathbb{F}_d, \mathcal{O}_{\mathbb{F}_d}(\alpha \Delta_0+\beta\Gamma))=0$ when $\alpha\ge-1$ and $\beta\ge \alpha d -1.$

\begin{lemma}\label{lemma:secdom-hirzebruch}
{\em
Let $a\ge1$ and $b\ge ad.$
At every point $p\in \ff_d$, the natural multiplication map
$$\begin{aligned}
&\left( H^0(\mathbb{F}_d, \mathcal{O}_{\mathbb{F}_d}(\Delta_0+d\Gamma) \otimes \mathcal{I}_{p/\mathbb{F}_d}) \otimes H^0(\mathbb{F}_d, \mathcal{O}_{\mathbb{F}_d}((a-1)\Delta_0+(b-d)\Gamma)) \right) \\
\oplus &\left( H^0(\mathbb{F}_d, \mathcal{O}_{\mathbb{F}_d}(\Gamma) \otimes \mathcal{I}_{p/\mathbb{F}_d}) \otimes H^0(\mathbb{F}_d, \mathcal{O}_{\mathbb{F}_d}(a\Delta_0+(b-1)\Gamma)) \right) \\
\to &H^0(\mathbb{F}_d, \mathcal{O}_{\mathbb{F}_d}(a\Delta_0+b \Gamma) \otimes \mathcal{I}_{p/\mathbb{F}_d})
\end{aligned}$$
is surjective.
Consequently, there is a surjective map
\begin{equation*}
M_{\Delta_0+d\Gamma}^{\oplus s}\oplus M_\Gamma^{\oplus t}\to M_{a\Delta_0+b\Gamma}
\end{equation*}
where $s=h^0(\mathbb{F}_d, \mathcal{O}_{\mathbb{F}_d}((a-1)\Delta_0+(b-d)\Gamma))$ and $t=h^0(\mathbb{F}_d, \mathcal{O}_{\mathbb{F}_d}(a\Delta_0+(b-1)\Gamma))$.
}
\end{lemma}

The following result by Chen, which was used in \cite{RoRo13} to study double covers of $\ff_0\simeq\pp^1\times\pp^1$, implies that $(\ff_d,D)$ is log algebraically hyperbolic if $b\ge a \ge 4$ and $b\ge ad+3.$

\begin{theorem}[{\cite[Corollary 1.12]{Che01}}]\label{theorem:hirzebruch-xi-chen}
{\em
Let $D=\bigcup D_k\subset \ff_d$, where each $D_k$ is a very general curve in a basepoint-free complete linear system, and assume $D\in |a\Delta_0 + b\Gamma|$ with $b\ge a.$ 
Then 
\begin{equation*}
2g(C)-2+\left|\nu^{-1}(D) \right| \geq \min(a-3,b-ad-2)\deg (C)
\end{equation*}
for all integral curves $C\neq D$. 
Here $g(C)$ denotes the geometric genus of $C$, $\nu: \widetilde{C} \to \ff_d$ denotes the induced map from the normalization of $C$, and $\deg(C)=(\Delta_0 + (d+1)\Gamma)\cdot C$.
}
\end{theorem}

We prove a different version of this result, which gives more refined information.
We treat the $d=0$ and $d>0$ cases separately in Propositions \ref{prop:hirzebruch} and \ref{prop:hirzebruch-f0}, respectively.

\begin{proposition}\label{prop:hirzebruch}
{\em
Let $d>0$, and let $D\subset \ff_d$ be a very general curve in the linear system $|a\Delta_0 + b\Gamma|$ where $a\ge 1$ and $b\ge ad$. 
Let $f:C\to \ff_d$ be a map from a smooth projective curve $C$ of genus $g$ that is birational onto its image, satisfies $f(C)\not\subset D\cup\Delta_0$, and meets $D$ in exactly $i$ distinct points. 
Then the curve $C$ must satisfy at least one of the following two inequalities:
\begin{enumerate}
\item $2g-2+i\ge ((a-3)\Delta_0 + (b-2d-2)\Gamma)\cdot C$
\item $2g-2+i\ge ((a-2)\Delta_0 + (b-d-3)\Gamma)\cdot C$
\end{enumerate}
Furthermore, if inequality (1) is satisfied as a strict equality, then the image $f(C)$ is limited to the following classes:
\begin{itemize}
\item When $d>1$, the fiber class $\Gamma$;
\item When $d=1$, the fiber class $\Gamma$ or the class $\Delta_0+\Gamma$.
\end{itemize}
}
\end{proposition}

\begin{proof}
We will use the setup and notation from \S\ref{sec:variational-argument}.
Recall from \eqref{eq:deg-log-normal} that at a general $b\in B$:
\[\deg N_{f_b/\ff_d}(\log D_{f,b}
)=2g-2+i-((a-2)\Delta_0+(b-d-2)\Gamma)\cdot C, \]
so we may assume that $\deg N_{f_b/\ff_d}(\log D_{f,b}
)<0.$
Since $T_{\ff_d}$ is pseudo-nef outside of $\Delta_0$, Proposition~\ref{prop:log-normal-degree-bound} and Lemma~\ref{lemma:secdom-hirzebruch} imply that we have one of the following two cases.

\textbf{Case 1:} There is a generically surjective map 
\begin{equation}
    \mu: f_b^*M_{\Delta_0+d\Gamma}\to N_{f_b/\ff_d}(\log D_{f,b}).
\end{equation}
Let us denote the image sheaf of $\mu$ by $Q$.
Arguing as in Proposition~\ref{prop:p2-horikawa}, we have $\deg Q\ge -(\Delta_0+d\Gamma)\cdot C_b$, which implies (1). 

Now, suppose equality occurs. 
Then we must have $\deg Q = -(\Delta_0+d\Gamma)\cdot C_b$, which forces the kernel $K$ of $\mu$ to have degree $0$.
Thus, $K$ is a subsheaf of the trivial bundle $\mathcal{O}_{C_b}^{\oplus d+2}$ of degree 0 and rank $d$. 
By Lemma~\ref{lemma:trivial-subbundle}, this implies that $K\simeq \mathcal{O}_{C_b}^{\oplus d}$, so $f_b^*M_{\Delta_0+d\Gamma}$ must possess at least $d$ global sections. 
By Proposition~\ref{prop:sections-of-kernel-bundles}, if the image of $C_b$ is an integral curve of class $\alpha \Delta_0+\beta \Gamma$, we have:
$$h^0(C_b, f_b^*M_{\Delta_0+d\Gamma}) = h^0(\mathbb{F}_d, \mathcal{I}_{f_b(C_b)/\mathbb{F}_d}(\Delta_0+d\Gamma)) = h^0(\mathbb{F}_d, \mathcal{O}_{\mathbb{F}_d}((1-\alpha)\Delta_0 + (d-\beta)\Gamma))$$
Since $f_b(C_b) \not\subset \Delta_0$, we must have $\alpha \ge 0$ and $\beta \ge ad$. 
For this dimension to be at least $d$, we are limited to two possibilities:
\begin{itemize}
\item $\alpha=0$ and $\beta=1$ (the class $\Gamma$), which yields $h^0(\mathbb{F}_d, \mathcal{O}_{\mathbb{F}_d}(\Delta_0 + (d-1)\Gamma)) = d$.
\item $\alpha=1$ and $\beta \le 1$ (the classes $\Delta_0$ or $\Delta_0+\Gamma$). 
Since $f_b(C_b) \not\subset \Delta_0$, the class cannot be $\Delta_0$. 
The class $\Delta_0+\Gamma$ yields $h^0(\mathbb{F}_d, \mathcal{O}_{\mathbb{F}_d}((d-1)\Gamma)) = d$.
Note that the class $\Delta_0+\Gamma$ is irreducible only when $d=1$.
\end{itemize}

\textbf{Case 2:} There is a generically surjective map 
\[\mu: f_b^*M_\Gamma\to N_{f_b/\ff_d}(\log D_{f,b}). \]
The image $Q$ of $\mu$ satisfies $\deg Q \ge -\Gamma\cdot C_b$. 
This immediately yields the inequality in (2).
\end{proof}

\begin{lemma}\label{lemma:trivial-subbundle}
{\em
Let $C$ be a smooth projective curve. 
If $K\subset \oo_C^{\oplus N}$ is a subsheaf of degree $0$ and rank $r<N$, then $K\simeq \oo_C^{\oplus r}$. 
}
\end{lemma}

\begin{proof}
We prove the statement by induction on the rank $r<N$. 
The base case $r=1$ follows from Schur's Lemma.
Assume $r>1$. 
Since $K\subset \oo_C^{\oplus N}$ is a degree $0$ subsheaf, it is a slope semistable subbundle of $\oo_C^{\oplus N}$. 
Moreover, Jordan-H\"older (JH) filtrations for $K$ and $\oo_C^{\oplus N}/K$ induce a JH filtration for $\oo_C^{\oplus N}$. 
By uniqueness of the JH factors, we see that each JH factor of $K$ is isomorphic to $\oo_C$. 
In particular, we have an inclusion $\oo_C\subset K$. 
Denote by $Q$ its quotient and consider the commutative diagram
\begin{equation}\label{eq:triangles}
\xymatrix{
0 \ar[r] & \oo_C  \ar[r] \ar@{=}[d] & K \ar[r] \ar@{^{(}->}[d] & Q \ar[r] \ar@{^{(}->}[d] & 0 \\
0  \ar[r] & \oo_C   \ar[r]^j & \oo_C^{N}\ar[r] &  R \ar[r] & 0,
}
\end{equation}
where $j:\oo_C\to \oo_C^{\oplus N}$ denotes the composition $\oo_C\subset  K\subset  \oo_C^{\oplus N}$. 
Since $j$ is given by a matrix of constant coefficients as $H^0(C,\oo_C)=\mathbb{C}$, we must have $R\simeq \oo_C^{\oplus N-1}$. 
In particular, we see that $Q$ is a torsion-free sheaf and so it is a vector bundle of degree $0$ and rank $r-1$. 
Thus, by induction $Q\simeq \oo_C^{\oplus r-1}$, and so it is left to prove that the extension element in $\operatorname{Ext}^1(Q,\oo_C)$ associated to the top sequence in \eqref{eq:triangles} is trivial. 
In \eqref{eq:triangles}, since the bottom sequence is split and the top sequence is pulled back from it via $Q\hookrightarrow R$, the functoriality of Ext implies that the top sequence is also split.
\end{proof}

\begin{remark}
    Proposition~\ref{prop:hirzebruch} yields several immediate consequences about pseudo-log Lang algebraic hyperbolicity of $(\ff_d,D)$ when $d>0$.
    First, when $D$ is a very general basepoint-free divisor in $|a\Delta_0+b\Gamma|$, it recovers the log algebraic hyperbolicity bounds $a\ge4$ and $b\ge ad+3$ established by Chen \cite{Che01}, after taking $\Delta_0$ into account.
    Furthermore, it provides new effective bounds for pseudo-log algebraic hyperbolicity. 
    For example, it implies that $(\mathbb{F}_d,D)$ is 
    pseudo-log algebraically hyperbolic outside of $\Delta_0$ when  
    $a\ge 4$ and $\begin{cases}
        b \ge ad & \text{when } d \ge 2\\
        b \ge 5 &\text{when } d=1
    \end{cases}$.
    Moreover, when $d=1$ and $a=b=4$, the above result implies that $(\ff_1, D)$ is pseudo-log Lang algebraically hyperbolic outside of $\Delta_0$ and the curves of class $\Delta_0+\Gamma$ that are bicontact to $D$ ($i=2$).
    This is equivalent to the fact that when $D_4\subset \pp^2$ is a very general quartic, $(\pp^2, D_4)$ is pseudo-log Lang algebraically hyperbolic outside of bitangent and flex lines to $D_4$.
\end{remark}

We omit the proof of the following result since it is similar to and simpler than Proposition~\ref{prop:hirzebruch}.

\begin{proposition}\label{prop:hirzebruch-f0}
{\em
Let $D\subset \ff_0$ be a very general curve in the linear system $|a\Delta_0 + b\Gamma|$ where $a\ge 1$ and $b\ge 0$. 
Let $f:C\to \ff_0$ be a map from a smooth projective curve $C$ of genus $g$ that is birational onto its image, satisfies $f(C)\not\subset D$, and meets $D$ in exactly $i$ distinct points. Then
\begin{enumerate}
    \item either $2g-2+i\ge ((a-3)\Delta_0 + (b-2)\Gamma)\cdot C$,
    \item or $2g-2+i\ge ((a-2)\Delta_0 + (b-3)\Gamma)\cdot C.$
\end{enumerate}
}
\end{proposition}

We apply the above results to the specific linear systems that produce Horikawa surfaces.

\begin{proposition}\label{prop:hirzebruch-horikawa}
{\em
Let $d>0$, and let $D\subset \ff_d$ be a very general curve in the linear system $|6\Delta_0+b\Gamma|$, where $b$ is an even integer satisfying $b\ge \max\{6d,6+4d\}$, or $(d,b)=(2,12)$.
For any map $f:C\to \ff_d$ from a smooth projective curve $C$ of genus $g$ that is birational onto its image, satisfies $f(C)\not\subset D\cup \Delta_0$, and meets $D$ in exactly $i$ distinct points,
we have
\begin{equation}
2g-2+i\ge \frac{1}{2}D\cdot C.    
\end{equation}
Moreover, equality is only possible when $f(C)$ is of class $\Gamma$. 
}  
\end{proposition}

\begin{proof}
First, assume $b\ge \max\{6d,6+4d\}$. We will treat the case $(d,b)=(2,12)$ separately at the end. 

Consider the two cases from the proof of Proposition~\ref{prop:hirzebruch} with $a=6$.

\textbf{Case 1:} We have $2g-2+i\ge (3\Delta_0 + (b-2d-2)\Gamma)\cdot C$, hence
\[2g-2+i-\frac{1}{2}D\cdot C \ge \left(\frac{1}{2}b-2d-2\right)\Gamma\cdot C. \]
Since $b\ge 6+4d$, the right-hand side is always non-negative, and it equals zero exactly when $f(C)$ is of class $\Gamma$.

\textbf{Case 2:} We have $2g-2+i\ge (4\Delta_0 + (b-d-3)\Gamma)\cdot C$, hence
\[2g-2+i-\frac{1}{2}D\cdot C \ge (\Delta_0 + \left(\frac{1}{2}b-d-3\right)\Gamma)\cdot C.\]
When $d\ge 4$ and $b\ge6d$, the right-hand side is always positive because $\Delta_0+(d+1)\Gamma$ is ample.
When $0\le d \le 3$ and $b\ge 6+4d$, the right-hand side is always non-negative because $\Delta_0+d\Gamma$ is nef, and it is zero exactly when the curve is $\Delta_0$. 
However, this does not occur becasue we assume $f(C)\not\subset \Delta_0$. 
Therefore, equality is never achieved in Case 2.

Now we take care of the case $(d,b)=(2,12).$

\textbf{Case 1:} We have $2g-2+i-\frac{1}{2}D\cdot C \ge 0$. Moreover, Proposition~\ref{prop:hirzebruch} implies that equality only occurs when $f(C)$ is of class $\Gamma$. 

\textbf{Case 2:}
We have 
\[2g-2+i-\frac{1}{2}D\cdot C \ge (\Delta_0 + \Gamma)\cdot C.\]
For any integral curve $C\ne\Delta_0$, we have $(\Delta_0 + \Gamma)\cdot C>0$. 
Hence, equality is not achieved.

\end{proof}

\begin{proposition}\label{prop:hirzebruch-horikawa-f0}
{\em
Let $D\subset \ff_0$ be a very general curve in the linear system $|6\Delta_0+b\Gamma|$, where $b$ is an even integer satisfying $b\ge6.$
For any map $f:C\to \ff_0$ from a smooth projective curve $C$ of genus $g$ that is birational onto its image, satisfies $f(C)\not\subset D$, and meets $D$ in exactly $i$ distinct points,
we have
\begin{equation}
2g-2+i\ge \frac{1}{2}D\cdot C.    
\end{equation}
Moreover, equality is only possible when $f(C)$ is of class $\Gamma$, or additionally of class $\Delta_0$ in the specific case where $b=6$.
}  
\end{proposition}

\begin{proof}
By Proposition~\ref{prop:hirzebruch-f0}, we have two cases.

\textbf{Case 1:} We have $$2g-2+i - \frac{1}{2}D\cdot C \ge \left(\frac{b}{2}-2\right)\Gamma\cdot C.$$
Since $b\ge 6$, the right-hand side is non-negative, and equals zero only if $f(C)$ is of class $\Gamma$.

\textbf{Case 2:} We have 
$$2g-2+i - \frac{1}{2}D\cdot C \ge \Delta_0\cdot C + \left(\frac{b}{2}-3\right)\Gamma\cdot C.$$
Since $b\ge 6$, the right-hand side is non-negative, and equals zero only if $b=6$ and $f(C)$ is of class $\Delta_0$.
\end{proof}

\begin{corollary}\label{cor:hirzebruch-horikawa-cover}
{\em
Let $d>0$, and let $\rho:S'\to \ff_d$ be a double cover branched over a very general curve $D$ in the linear system $|6\Delta_0 + b\Gamma|$, where $b$ is an even integer satisfying either:
\begin{enumerate}
    \item $b\ge \max\{6d, 6+4d\}$, or
    \item $d=2$ and $b=12$.
\end{enumerate}
Then, outside of $\rho^{-1}(\Delta_0)$, $S'$ does not contain rational curves.
Furthermore, if $C'\subset S'$ is an integral curve with $g(C')=1$ not mapping to $\Delta_0$, then $C'=\rho^{-1}(\Gamma)$ for a fiber $\Gamma\subset \ff_d$ that is tangent to $D$.
In particular, $S'$ is pseudo-Lang algebraically hyperbolic.
}
\end{corollary}

\begin{proof}
This follows immediately from the proof of Lemma~\ref{lemma:log-hyp-to-cover-hyp} and Proposition~\ref{prop:hirzebruch-horikawa}.
Furthermore, if $C'\subset S'$ is an elliptic curve not mapping to $\Delta_0$ such that $\rho(C')$ has class $\Gamma$, then Remark~\ref{rem:elliptic-curves-double cover} implies that $\rho(C')$ has to be tangent to $D$ at one point.
\end{proof}

\begin{remark}
Let us consider $\rho^{-1}(\Delta_0)$ in Corollary~\ref{cor:hirzebruch-horikawa-cover}.
$\Delta_0$ intersects $D$ transversally at $b-6d$ points, so $\rho^{-1}(\Delta_0)$ contains rational or elliptic curves only when $b-6d=0,2$ or 4. 
With regards to Horikawa surfaces, this means that $\rho^{-1}(\Delta_0)$ consists of:
\begin{itemize}
    \item 2 rational curves when $S'$ is very general of $n=3$ and type $(2')$, or $n\ge6$ and type $\left(\frac{n+3}{3} \right).$
    \item 1 rational curve when $S'$ is very general of $n=5$ and type $(2)$, or $n\ge6$ and type $\left(\frac{n+1}{3} \right).$
    \item 1 elliptic curve when $S'$ is very general of $n=4$ and type $(1)$, or $n\ge6$ and type $\left(\frac{n-1}{3} \right).$
\end{itemize}
\end{remark}

\begin{corollary}\label{cor:hirzebruch-f0-horikawa-cover}
{\em
Let $\rho:S'\to \ff_0$ be a double cover branched over a very general curve $D$ in the linear system $|6\Delta_0 + b\Gamma|$, where $b\ge6$ is an even integer.
Then, $S'$ does not contain rational curves.
Furthermore, if $C'\subset S'$ is an integral curve with $g(C')=1$,
then $C'=\rho^{-1}(\Gamma)$ for a curve in the fiber class $\Gamma$ that is tangent to $D$ (or additionally, in the specific case where $b=6$, $C'$ may be the pullback of a curve in class $\Delta_0$ that is tangent to $D$).
In particular, $S'$ is pseudo-Lang algebraically hyperbolic.
}
\end{corollary}

Many Horikawa surfaces are double covers of $\ff_d$ with $d>0$ branched over a divisor that consists of the fixed curve $\Delta_0$ plus a moving curve, so we record the analogous results in these cases.

\begin{proposition}\label{prop:hirzebruch-reducible}
{\em
Let $d>0$, and let $D'\subset \ff_d$ be a very general curve in the linear system $|a\Delta_0 + b\Gamma|$ where $a\ge 1$ and $b\ge ad$.
Let $D=\Delta_0 + D'$. 
Let $f:C\to \ff_d$ be a map from a smooth projective curve $C$ of genus $g$ that is birational onto its image, satisfies $f(C)\not\subset D$, and meets $D$ in exactly $i$ distinct points. 
Then:
\begin{enumerate}
    \item either $2g-2+i\ge ((a-2)\Delta_0 + (b-2d-2)\Gamma)\cdot C$,
    \item or $2g-2+i\ge ((a-1)\Delta_0 + (b-d-3)\Gamma)\cdot C.$
\end{enumerate}
In Case (1), equality is achieved only when $C$ has the following classes:
\begin{itemize}
    \item When $d>1$, $\Gamma$;
    \item When $d=1$, $\Gamma$ or $\Delta_0+\Gamma$.
\end{itemize}

}
\end{proposition}

\begin{remark}
The above proposition implies that when $D'$ is a very general curve in $|a\Delta_0+b\Gamma|$ with $a\ge3$ and $b\ge ad+4$, then $(\ff_d,\Delta_0+D')$ is log algebraically hyperbolic.    
\end{remark}

\begin{proposition}\label{prop:hirzebruch-reducible-horikawa}
{\em
Let $d>0$, and let $D'\subset \ff_d$ be a very general curve in the linear system $|a\Delta_0+b\Gamma|$ where $b$ is an even integer and $a,b,d$ fall in one of the following cases: 
\begin{enumerate}
    \item $a=5$ and $b\ge \max\{5d,6+4d\}$; or
    \item $a=5$ and $(d,b)=(3,16)$ or $(4,20)$; or
    \item $a=7$ and $(d,b)=(2,14)$.
\end{enumerate}
Let $D=\Delta_0+D'.$ 
For any map $f:C\to \ff_d$ from a smooth projective curve $C$ of genus $g$ that is birational onto its image, satisfies $f(C)\not\subset D$, and meets $D$ in exactly $i$ distinct points, we have:
\begin{equation*}
2g-2+i\ge \frac{1}{2}D\cdot C.    
\end{equation*}
Moreover, equality is only possible when $a=5$ and $f(C)$ is of class $\Gamma$.
}    
\end{proposition}

\begin{proof}
First, assume $D'\in |5\Delta_0+b\Gamma|$ with $b\ge \max\{5d,6+4d\}$.
We consider the cases in the proof of Proposition~\ref{prop:hirzebruch-reducible}.

\textbf{Case 1:} We have $2g-2+i\ge (3\Delta_0 + (b-2d-2)\Gamma)\cdot C$, so
\[2g-2+i-\frac{1}{2}D\cdot C \ge \left(\frac{1}{2}b-2d-2\right)\Gamma\cdot C. \]
Since $b\ge 6+4d$, the right-hand side is always non-negative, and it is zero only when $f(C)$ is of class $\Gamma.$

\textbf{Case 2:} We have $2g-2+i\ge (4\Delta_0 + (b-d-3)\Gamma)\cdot C$, so
\[2g-2+i-\frac{1}{2}D\cdot C \ge (\Delta_0 + \left(\frac{1}{2}b-d-3\right)\Gamma)\cdot C.\]
When $d\ge7$ and $b\ge5d$, the right-hand side is always positive because $\Delta_0+(d+1)\Gamma$ is ample.
When $0 \le d \le 6$ and $b\ge 6+4d$, the right-hand side is always non-negative because $\Delta_0+d\Gamma$ is nef, and it is zero only when the curve maps to $\Delta_0$. Thus, equality is not achieved in Case 2.

Next, assume $a=5$ and $(d,b)=(3,16)$ or $(4,20)$.
Case 1 goes similarly as above to yield the desired inequality and to show that equality holds only when $f(C)$ has class $\Gamma$.
In Case 2, we have $2g-2+i-\frac{1}{2}D\cdot C \ge (\Delta_0 + (d-1)\Gamma)\cdot C$.
The right-hand side is always positive except when the curve maps to $\Delta_0.$

Finally, assume $a=7$ and $(d,b)=(2,14)$. 
In Case 1, we have $2g-2+i-\frac{1}{2}D\cdot C \ge (\Delta_0+\Gamma)\cdot C.$
In Case 2, we have $2g-2+i-\frac{1}{2}D\cdot C \ge 2(\Delta_0+\Gamma)\cdot C.$
When $f(C)\ne \Delta_0$, the right-hand sides of both of these inequalities are positive.
\end{proof}

\begin{corollary}\label{cor:hirzebruch-reducible-horikawa-cover}
{\em
Let $d>0$, and let $\rho:S'\to \ff_d$ be a double cover branched over $D=\Delta_0+D'$ where $D'$ is a very general curve in the linear system $|a\Delta_0 + b\Gamma|$ where $b$ is an even integer.

Suppose $a=5$ and we are in one of the following cases:
\begin{enumerate}
    \item $b\ge\max\{5d,6+4d\}$; or 
    \item $(d,b)=(3,16)$ or $(4,20)$.
\end{enumerate}
Then the only rational curve in $S'$ is the pullback $\rho^{-1}(\Delta_0)$.
Furthermore, if $C'\subset S'$ is an integral curve with $g(C')=1$, then $C'=\rho^{-1}(\Gamma)$ for a fiber $\Gamma\subset \ff_d$ that is either tangent to $D'$ at a point or intersects $D$ at a node.
In particular, $S'$ is pseudo-Lang algebraically hyperbolic.

Additionally, for $a=7$ and $(d,b)=(2,14)$, $S'$ pseudo-Lang algebraically hyperbolic outside of $\rho^{-1}(\Delta_0)$.
}
\end{corollary}

\subsection{Summary for very general Horikawa surfaces of the first kind}\label{sec:summary-1st}
To summarize the results thus far, we have shown that a very general Horikawa surface of the first kind with $n\ge 2$ is pseudo-Lang algebraically hyperbolic. Specifically, the distribution of rational or elliptic curves across the moduli space is as follows:
\begin{itemize}
\item \textbf{When $n=2$:} A very general surface (type $(\infty)$) contains $0$ rational curves and $1320$ elliptic curves by Corollary~\ref{cor:p2-horikawa-cover}.

\item \textbf{When $n=3$:} A very general surface in the big stratum (type $(0)$) contains $0$ rational curves and $2\cdot(10n+30)=120$ elliptic curves by Corollary~\ref{cor:hirzebruch-f0-horikawa-cover}.
A very general surface in the small stratum (type $(2')$) contains $2$ rational curves and $60$ elliptic curves by Corollary~\ref{cor:hirzebruch-horikawa-cover}.

\item \textbf{When $n=4$:} A very general surface in the big stratum (type $(1)$) contains 0 rational curves and $(10n+30)+1=71$ elliptic curves by Corollary~\ref{cor:hirzebruch-horikawa-cover}.
A very general surface in the small stratum (type $(3')$) contains 2 rational curves (the pullback of $\Delta_0$ and the exceptional divisor of a node) and $68+1=69$ elliptic curves by Corollary~\ref{cor:hirzebruch-reducible-horikawa-cover}.

\item \textbf{When $n=5$:} The moduli space has two components. 
\begin{itemize}
    \item For the first component, a very general surface in the big stratum (type $(0)$) contains 0 rational curves and $10n+30=80$ elliptic curves. 
    A very general surface in the small stratum (type $(2)$) contains 1 rational curve and $10n+30=80$ elliptic curves. See Corollaries~\ref{cor:hirzebruch-horikawa-cover} and \ref{cor:hirzebruch-f0-horikawa-cover}.
    \item For the second component, a very general surface in the big stratum (type $(\infty)$) is Demailly algebraically hyperbolic, hence it contains no rational or elliptic curves. 
    A very general surface in the small stratum (type $(4')$) contains 1 rational curve (the pullback of $\Delta_0$) and 80 elliptic curves by Corollary~\ref{cor:hirzebruch-reducible-horikawa-cover}.
\end{itemize}

\item \textbf{When $4\mid n-1$ and $n\ge 9$:} The moduli space has two components. 
\begin{itemize}
    \item For the first component, a very general surface in the larger strata (types $(0)$ down to $(d)$ with $d\le \frac{n+3}{3}$) contains:
    \begin{itemize}
        \item 0 rational curves and $10n+30$ elliptic curves when $d<\frac{n-1}{3}$. 
        \item 0 rational curves and $(10n+30)+1=10n+31$ elliptic curves when $d=\frac{n-1}{3}$.
        \item 1 rational curve and $10n+30$ elliptic curves when $d=\frac{n+1}{3}$. 
        \item 2 rational curves and $10n+30$ elliptic curves when $d=\frac{n+3}{3}$.
    \end{itemize}
    The above follow from Corollaries~\ref{cor:hirzebruch-horikawa-cover} and \ref{cor:hirzebruch-f0-horikawa-cover}.
    In the smaller strata where $d>\frac{n+3}{3}$, a very general surface contains $n+4-2d$ rational curves (from $\Delta_0$ plus $n+3-2d$ nodes on $D$) and $(8n+4d+24)+(n+3-2d)=9n+2d+27$ elliptic curves by Corollary~\ref{cor:hirzebruch-reducible-horikawa-cover}.
    \item For the second component (type $\left(\frac{n+3}{2}\right)$), a very general surface contains exactly 1 rational curve (from $\Delta_0$) and $8n+4d+24=10n+30$ elliptic curves by Corollary~\ref{cor:hirzebruch-reducible-horikawa-cover}. 
\end{itemize}

\item \textbf{When $4\nmid n-1$ and $n\ge 6$:} The moduli space consists of a single component, and the numbers are the same as the first component in the previous case.
\end{itemize}

This is summarized in Table~\ref{table:horikawa-count-1}.
In all cases described above, the geometrically elliptic curves arise either as $\rho^{-1}(\Delta_0)$ or as pullbacks of rational curves in the fiber class $\Gamma$ (and additionally in class $\Delta_0$ when $d=0$ and $n=3$) that are either simply tangent to the branch locus $D$ or intersects it at a node. 
The rational curves arise strictly as components of $\rho^{-1}(\Delta_0)$, or as the exceptional divisors of blown-up singular points on the ramification loci $D$.

\begin{table}[htbp]
\centering
\renewcommand{\arraystretch}{1.3}
\begin{tabular}{clccl}
\toprule
\textbf{Component} & \textbf{Stratum} & \textbf{Condition} & \textbf{Rational} & \textbf{Elliptic} \\
\midrule
$n=2$ & Type $(\infty)$ & -- & $0$ & $1320$ \\
\midrule
\multirow{2}{*}{$n=3$} 
& Type $(0)$ & -- & $0$ & $120$ \\
& Type $(2')$ & -- & $2$ & $60$ \\
\midrule
\multirow{2}{*}{$n=4$} 
& Type $(1)$ & -- & $0$ & $71$ \\
& Type $(3')$ & -- & $2$ & $69$ \\
\midrule
\multirow{2}{*}{$n=5$} 
& Type $(0)$  & -- & $0$ & $80$  \\
& Type $(2)$ & -- & $1$ & $80$  \\
\midrule
\multirow{2}{*}{$n=5$}  & Type $(\infty)$ & -- & $0$ & $0$  \\
& Type $(4')$ & -- & $1$ & $80$  \\
\midrule
\multirow{2}{*}{\parbox{2cm}{\centering $n \ge 9$\\ $4 \mid n-1$}} 
& Type $(d)$  & $d< \frac{n-1}{3}$ & $0$ & $10n+30$  \\
& Type $(d)$  & $d= \frac{n-1}{3}$ & $0$ & $10n+31$  \\
& Type $(d)$  & $d= \frac{n+1}{3}$ & $1$ & $10n+30$  \\
& Type $(d)$  & $d= \frac{n+3}{3}$ & $2$ & $10n+30$  \\
& Type $(d)$ & $\frac{n+3}{3}<d\le \frac{n-1}{2}$ & $n+4-2d$ & $9n+2d+27$ \\
\midrule
\parbox{2cm}{\centering $n \ge 9$\\ $4 \mid n-1$}& Type $\left(\frac{n+3}{2}\right)$ & -- & $1$ & $10n+30$  \\
\midrule
\multirow{2}{*}{\parbox{2cm}{\centering $n \ge 6$\\ $4 \nmid n-1$}} 
& Type $(d)$  & $d< \frac{n-1}{3}$ & $0$ & $10n+30$  \\
& Type $(d)$  & $d= \frac{n-1}{3}$ & $0$ & $10n+31$  \\
& Type $(d)$  & $d= \frac{n+1}{3}$ & $1$ & $10n+30$  \\
& Type $(d)$  & $d= \frac{n+3}{3}$ & $2$ & $10n+30$  \\
& Type $(d)$ & $\frac{n+3}{3}<d$ & $n+4-2d$ & $9n+2d+27 $ \\
\bottomrule
\end{tabular}
\vspace{.2in}
\caption{Number of rational and elliptic curves on very general Horikawa surfaces of the first kind in each strata and component of moduli.
Strata are listed in decreasing order of dimension.}
\label{table:horikawa-count-1}
\end{table}

It also follows from Corollary~\ref{cor:hirzebruch-reducible-horikawa-cover} that a very general type $(2^*)$ surface, which is a Horikawa surface of the second kind with $n=5$, contains no rational curves (the curve pulled back from $\Delta_0$ is contracted by $\pi$) and no elliptic curves.

\subsection{Blowup of Hirzebruch surfaces at two points on the same fiber, neither on $\Delta_0$}\label{sec:hirzebruch-2pt}

In this subsection, $q: W \to \ff_d$ denotes the blowup of $\ff_d$ at two distinct points, $x$ and $y$, lying on the same fiber $\Gamma_0$. 
Let us denote by $\widetilde{\Gamma}$ the proper transform of this fiber, and by $E_x$ (resp. $E_y$) the exceptional curve over $x$ (resp. $y$).
When $d > 0$, we assume that neither $x$ nor $y$ lies on $\Delta_0$. 
When $d = 0$, we denote the curves of class $\Delta_0$ passing through $x$ and $y$ by $\Delta_x$ and $\Delta_y$, respectively.

For a general point in each stratum of Horikawa surfaces of the second kind, specifically, those of type $(d)$ for $n \ge 4$ with $d \le \frac{n+4}{3}$, the corresponding surface is birational to a double cover of $W$. 
Moreover, $x$ and $y$ are generic points on $\Gamma_0$, and the branch locus is a disjoint union of the strict transform $\widetilde{\Gamma}$ and a general divisor.

Recall that the canonical class of $W$ is given by
$$K_W = -2q^*\Delta_0 - (d+2) q^*\Gamma + E_x + E_y.$$

First, we establish positivity results for the log tangent bundle $T_W(-\log \widetilde{\Gamma})$.

\begin{proposition}[Case $d>0$]\label{prop:nefness-2points-d>0}
{\em
Let $d>0$ and let $q:W\to \ff_d$ be the blowup of $\ff_d$ at two distinct points $x,y$ lying on the same fiber $\Gamma_0$, neither on $\Delta_0$.  
Let $\widetilde{\Gamma}$ be the proper transform of the fiber $\Gamma_0$, and $E_x,E_y$ the exceptional divisors mapping to $x,y$ respectively. 
Assume that there is a non-constant map $f:C\to W$ from a smooth projective curve and a surjective map  $$f^*T_W(-\log \widetilde{\Gamma})\to L,$$ where $L$ is a line bundle with $\deg L<0$. Then $f(C)$ is either $E_x,E_y$ or $\widetilde{\Delta}_0$.
}
\end{proposition}

\begin{proof}
First we will show that if $f(C) \notin \{E_x, E_y, \widetilde{\Delta}_0, \widetilde{\Gamma}\}$, then any quotient line bundle $L$ of $f^*T_W(-\log \widetilde{\Gamma})$ must have $\deg L \geq 0$.
Let $D = \widetilde{\Gamma} + E_x + E_y + \widetilde{\Delta}_0$, which is a simple normal crossing divisor, so we have the short exact sequence:
$$0 \to T_W(-\log D) \to T_W(-\log \widetilde{\Gamma}) \to \mathcal{O}_{\widetilde{\Delta}_0}(\widetilde{\Delta}_0) \oplus \mathcal{O}_{E_x}(E_x) \oplus \mathcal{O}_{E_y}(E_y) \to 0$$
Since $f(C)$ is not contained in $\widetilde{\Delta}_0, E_x,$ or $E_y$, the quotient $L$ induces a sub-line bundle $L' \subseteq L$ which is a quotient of $f^*T_W(-\log D)$. 
Consequently, $\deg L \geq \deg L'$. 
It thus suffices to prove that $f^*T_W(-\log D)$ has no line bundle quotients of negative degree.

Let $p: W \to \mathbb{P}^1$ be the composition of the blowup $q: W \to \mathbb{F}_d$ with the ruling $\mathbb{F}_d \to \mathbb{P}^1$. 
The differential induces a morphism $\varphi: T_W(-\log D) \to p^*T_{\mathbb{P}^1}$. 
Since $p$ contracts $E_x, E_y,$ and $\widetilde{\Gamma}$ to a single point $t \in \mathbb{P}^1$, the image of $\varphi$ vanishes at $t$, yielding a surjection onto $p^*(T_{\mathbb{P}^1}(-t)) \simeq p^*\mathcal{O}_{\mathbb{P}^1}(1)$. 
This gives a short exact sequence of vector bundles, where $K$ is the induced kernel:
$$0 \to K \to T_W(-\log D) \to p^*\mathcal{O}_{\mathbb{P}^1}(1) \to 0$$
Since $\det T_W(-\log D) = -K_W - D$, $K_W = q^*(-2\Delta_0 - (d+2)\Gamma) + E_x + E_y$ and $D = q^*(\Delta_0 + \Gamma)$, a quick computation yields:
$$K \simeq \mathcal{O}_W(q^*(\Delta_0 + d\Gamma) - E_x - E_y).$$
$L'$ is a quotient of the extension $f^*T_W(-\log D)$, so its degree is bounded below by the minimum of the degrees of the quotients of the outer terms. 
It suffices to check that $K \cdot C \geq 0$ for any curve $f(C) \notin \{E_x, E_y, \widetilde{\Delta}_0, \widetilde{\Gamma}\}$.
Suppose $f(C)$ has class $a\widetilde{\Delta}_0+b\widetilde{\Gamma}+c_xE_x+c_yE_y$, then $K\cdot C = c_x+c_y -b\ge b-a$, which is non-negative because $b\ge ad.$

Finally, if $f(C) = \widetilde{\Gamma}$, we note that $T_W(-\log \widetilde{\Gamma})\vert{}_{\widetilde{\Gamma}} \simeq T_{\widetilde{\Gamma}} \oplus \mathcal{O}_{\widetilde{\Gamma}} \simeq \mathcal{O}_{\widetilde{\Gamma}}(2) \oplus \mathcal{O}_{\widetilde{\Gamma}}$, so its quotients have non-negative degree. 
    
\end{proof}

\begin{proposition}[Case $d=0$]\label{prop:nefness-2points-d=0}
{\em
Let $q:W\to \ff_0$ be the blow-up of $\ff_0$ at two distinct points $x,y$ lying on the same fiber $\Gamma_0$.
Let $\widetilde{\Gamma}$ be the proper transform of the fiber $\Gamma_0$, $E_x,E_y$ the exceptional divisors mapping to $x,y$ respectively, and $\widetilde{\Delta}_x, \widetilde{\Delta}_y$ the proper transform of $\Delta_x,\Delta_y$.
Assume that there is a non-constant map $f:C\to W$ from a smooth projecitve curve and a surjective map  $$f^*T_W(-\log \widetilde{\Gamma})\to L,$$ where $L$ is a line bundle with $\deg L<0$. Then $f(C)$ is either $E_x,E_y,\widetilde{\Delta}_x,$ or $\widetilde{\Delta}_y.$
}
\end{proposition}

\begin{proof}
The proof is very similar to Proposition~\ref{prop:nefness-2points-d>0}. 
Assume that $f(C) \notin \{E_x, E_y, \widetilde{\Delta}_x, \widetilde{\Delta}_y, \widetilde{\Gamma}\}$. 
We will verify that $\deg L\geq 0$.  

Let $D = \widetilde{\Gamma} + E_x + E_y + \widetilde{\Delta}_x + \widetilde{\Delta}_y$. 
Since $D$ is a simple normal crossing divisor, there is again a short exact sequence:
$$0 \to T_W(-\log D) \to T_W(-\log \widetilde{\Gamma}) \to \mathcal{O}_{E_x}(E_x) \oplus \mathcal{O}_{E_y}(E_y) \oplus \mathcal{O}_{\widetilde{\Delta}_x}(\widetilde{\Delta}_x) \oplus \mathcal{O}_{\widetilde{\Delta}_y}(\widetilde{\Delta}_y) \to 0$$
Since $f(C)$ is not contained in the support of the quotient, the induced map from $f^*T_W(-\log D)$ to $L$ is nonzero, and its image is a sub-line bundle $L' \subseteq L$. 
This means $\deg L \geq \deg L'$, and it thus suffices to prove that $T_W(-\log D)$ is a nef vector bundle on $W$.

Let $p: W \to \mathbb{P}^1$ be the composition of the blowup $q: W \to \mathbb{F}_0$ followed by the projection $\mathbb{F}_0 \to \mathbb{P}^1$ that contracts the fibers $\Delta_0$. 
Let $t, s \in \mathbb{P}^1$ be the images of $x$ and $y$ under this projection. The differential map $T_W \to p^*T_{\mathbb{P}^1}$ induces a morphism $\varphi: T_W(-\log D) \to p^*T_{\mathbb{P}^1}$, whose image is the sheaf of tangent vectors on $\mathbb{P}^1$ vanishing at $\{t, s\}$. 
Since $p^*T_{\mathbb{P}^1} \simeq p^*\mathcal{O}_{\mathbb{P}^1}(2)$, this image is $p^*(T_{\mathbb{P}^1}(-t-s)) \simeq p^*\mathcal{O}_{\mathbb{P}^1} \simeq \mathcal{O}_W$. 
We thus obtain a short exact sequence:
$$0 \to K \to T_W(-\log D) \to \mathcal{O}_W \to 0$$
A computation as in the $d>0$ case yields:
$$K \simeq \mathcal{O}_W(q^*\Gamma).$$
Therefore, $T_W(-\log D)$ must be nef because it is an extension of two nef line bundles.

Finally, as in the $d>0$ case, if $f(C) = \widetilde{\Gamma}$, we have $T_W(-\log \widetilde{\Gamma})\vert{}_{\widetilde{\Gamma}} \simeq \mathcal{O}_{\widetilde{\Gamma}}(2) \oplus \mathcal{O}_{\widetilde{\Gamma}}$. 
\end{proof}

The following vanishing result is needed to find section-dominating collections of line bundles on an open set of $W$.

\begin{lemma}\label{lemma:h1-vanishing-hirzebruch2}
{\em
Let $d\ge 0$ and let $q:W\to \ff_d$ be the blowup of $\ff_d$ at two distinct points $x,y$ lying on the same fiber $\Gamma_0$, and neither on $\Delta_0$ when $d>0$. Let $$L=\oo_W(a\cdot q^*\Delta_0 + b \cdot q^* \Gamma - c_x\cdot E_x-c_y\cdot E_y)$$ with $a,b,c_x,c_y\ge 0$. If $a\ge c_x+c_y-1$ and $b\ge ad+\max\{c_x,c_y\}-1$, then $H^1(L)=0.$
}
\end{lemma}

\begin{proof}
Assume without loss of generality that $c_x\ge c_y.$

First, we show by an induction argument that it suffices to show that $H^1$ vanishes for the line bundle of $a\cdot q^*\Delta_0 + (b-c_y) \cdot q^* \Gamma - (c_x-c_y)\cdot E_x$.
For $0 \le k \le c_y,$ we denote by $L_k$ the line bundle of $a\cdot q^*\Delta_0 + (b-k) \cdot q^* \Gamma - (c_x-k)\cdot E_x-(c_y-k)\cdot E_y$.
For each $0 \le k \le c_y-1$, restricting $L_k$ to $\widetilde{\Gamma}$ gives the short exact sequence 
\begin{equation*}
    0 \to L_{k+1} \to L_k \to L_k\vert_{\widetilde{\Gamma}}\to 0.
\end{equation*}
Since $L_k\vert_{\widetilde{\Gamma}}\simeq \oo_{\pp^1}(a-c_x-c_y+2k),$ $H^1(\widetilde{\Gamma}, L_k\vert_{\widetilde{\Gamma}})=0$ when $a-c_x-c_y+2k\ge -1$. 
This holds under our assumption that $a\ge c_x + c_y -1.$
Since $H^1(W, L_{k+1})\twoheadrightarrow H^1(W,L_k)$, by induction it is enough to show that $H^1(W, L_{c_y
})=0.$

Now, a similar induction argument shows that it suffices to show that $H^1$ vanishes for the line bundle of $a\cdot q^*\Delta_0 + (b-c_x) \cdot q^* \Gamma$. For $c_y\le k \le c_x$, we denote by $M_k$ the line bundle of $a\cdot q^*\Delta_0 + (b-k) \cdot q^* \Gamma - (c_x-k)\cdot E_x$.
For each $c_y \le k \le c_x-1$, restricting $M_k$ to $\widetilde{\Gamma}$ gives 
\begin{equation*}
    0 \to M_{k+1}\otimes \oo_W(E_y) \to M_k \to M_k\vert_{\widetilde{\Gamma}}\to 0.
\end{equation*}
Since $M_k\vert_{\widetilde{\Gamma}}\simeq \oo_{\pp^1}(a-c_x+k)$ and we are assuming that $a\ge c_x + c_y -1$, it follows that $H^1(\widetilde{\Gamma},M_k\vert_{\widetilde{\Gamma}})=0$ and $H^1(W, M_{k+1}\otimes \oo_W(E_y))\twoheadrightarrow H^1(W, M_k).$
Upon restricting $M_{k+1}\otimes \oo_W(E_y)$ to $E_y$, we obtain 
\begin{equation*}
    0 \to M_{k+1}\to M_{k+1}\otimes \oo_W(E_y)\to M_{k+1}\otimes \oo_W(E_y)\vert_{E_y}\to 0.
\end{equation*}
Similar as above, we see that $M_{k+1}\otimes \oo_W(E_y)\vert_{E_y}\simeq \oo_{\pp^1}(-1)$, so we have a surjection $H^1(W,M_{k+1})\twoheadrightarrow H^1(W, M_k)$, and our claim holds by induction.

Finally, $H^1(W, \oo_W(a\cdot q^*\Delta_0 + (b-c_x)\cdot q^*\Gamma))\simeq H^1(\ff_d, \oo_{\ff_d}(a\cdot \Delta_0 + (b-c_x)\cdot \Gamma))$ vanishes when $b\ge ad+c_x-1.$
\end{proof}

\begin{lemma}[cf. Example 2.6 in \cite{CoRi23}]
\label{lemma:secdom-hirzebruch2}
{\em
Let $d\ge 0$ and let $q:W\to \ff_d$ be the blowup of $\ff_d$ at two distinct points $x,y$ lying on the same fiber $\Gamma_0$, and neither on $\Delta_0$ when $d>0$. Let $$L=\oo_W(a\cdot q^*\Delta_0 + b \cdot q^* \Gamma - c_x\cdot E_x-c_y\cdot E_y)$$ be a line bundle with $c_x,c_y\ge 0$, $a\ge c_x+c_y$ and $b\ge ad+\max\{c_x,c_y\}$.
At every point $p\in W\setminus q^{-1}(\Gamma_0)$, there is a surjective map
\begin{equation*}
\begin{aligned}
&\left( H^0(W, \mathcal{O}_{W}(q^*(\Delta_0+d\Gamma)) \otimes \mathcal{I}_{p/W}) \otimes H^0(W, \mathcal{O}_{W}((a-1)q^*\Delta_0 + (b-d)q^*\Gamma-c_xE_x-c_yE_y)) \right) \\
&\oplus \left( H^0(W, \mathcal{O}_{W}(q^*\Gamma) \otimes \mathcal{I}_{p /W}) \otimes H^0(W, \mathcal{O}_{W}(aq^*\Delta_0 + (b-1)q^*\Gamma-c_xE_x-c_yE_y)) \right) \\
&\to H^0(W,L \otimes \mathcal{I}_{p /W}).
\end{aligned}
\end{equation*}
In other words, outside of $q^{-1}(\Gamma_0)$, there is a surjective map
\begin{equation*}
M_{q^*(\Delta_0+d\Gamma)}^{\oplus s}\oplus M_{q^*\Gamma}^{\oplus t}\to M_L.
\end{equation*}
}
\end{lemma}

\begin{proof}
Let $p$ be a point in $W\setminus q^{-1}(\Gamma_0)$, and let $\widetilde{\Gamma}_p$ denote the strict transform of the fiber through $q(p).$
There exists a section $s_p$ of $\oo_W(q^*(\Delta_0+d\Gamma))$ that, upon restricting to the curve $\widetilde{\Gamma}_p$, vanishes only at $p$.
Consider the short exact sequence obtained by restricting
$$L' := \oo_W((a-1) q^*\Delta_0 + (b-d) q^* \Gamma - c_x E_x - c_y E_y)$$
to the curve $\widetilde{\Gamma}_p$:
\begin{equation*}
    0 \to L'\otimes \oo_W(-q^*\Gamma) \to L' \to \oo_{\pp^1}(a-1)\to 0.
\end{equation*}
Note that $L'\otimes \oo_W(-q^*\Gamma) \simeq \oo_W((a-1) q^*\Delta_0 + (b-d-1) q^* \Gamma - c_x E_x - c_y E_y)$.
By Lemma~\ref{lemma:h1-vanishing-hirzebruch2} and our assumptions on the coefficients $a,b,c_x,c_y$, $L'\otimes \oo_W(-q^*\Gamma)$ has vanishing $H^1$, and so we obtain a surjection
\begin{equation*}
    H^0(W,L')\to H^0(\pp^1, \oo_{\pp^1}(a-1)).
\end{equation*}
This means that there exist sections $s_1,\ldots,s_a$ of $L'$ whose restrictions to $\widetilde{\Gamma}_p$ span $H^0(\pp^1, \oo_{\pp^1}(a-1))$. 
Therefore, the products $s_ps_1,\ldots,s_ps_a$ span all sections in $H^0(W, L\otimes \mathcal{I}_{p/W})$, modulo the sections that vanish along the entire fiber $\widetilde{\Gamma}_p$. 
These remaining sections are generated by the second term in our direct sum, as a section of $\mathcal{O}_W(q^*\Gamma)$ vanishing at $p$ must vanish on the entirety of $\widetilde{\Gamma}_p$.
\end{proof}

The previous results allow us to study the curves on $W$ besides $\widetilde{\Gamma}, \widetilde{\Delta}_0, E_x$ and $E_y$ when $d>0$, and $\widetilde{\Gamma}, \widetilde{\Delta}_x,\widetilde{\Delta}_y,E_x$ and $E_y$ when $d=0.$

\begin{proposition}\label{prop:hirzebruch2}
{\em
Let $d> 0$ and let $q:W\to \ff_d$ be the blowup of $\ff_d$ at two distinct points $x,y$ lying on the same fiber $\Gamma_0$, and neither on $\Delta_0$. 
Let $D'\subset W$ be a very general curve in the linear system $$|a\cdot q^*\Delta_0 + b\cdot q^*\Gamma - c_x\cdot E_x - c_y\cdot E_y|$$ with $c_x,c_y\ge 0$, $a\ge c_x+c_y$ and $b\ge ad+\max\{c_x,c_y\}$, and let $D=\widetilde{\Gamma}+D'.$
Let $f:C\to W$ be a map from a smooth projective curve $C$ of genus $g$ that is birational onto its image, satisfies $f(C)\not\subset D\cup E_x\cup E_y\cup\widetilde{\Delta}_0$,
and meets $D$ in exactly $i$ distinct points. 
Then
\begin{enumerate}
    \item either $2g-2+i\ge ((a-3) q^*\Delta_0 + (b-2d-1) q^*\Gamma-c_x E_x-c_yE_y)\cdot C$,
    \item or $2g-2+i\ge ((a-2) q^*\Delta_0 + (b-d-2) q^*\Gamma-c_x E_x-c_yE_y)\cdot C.$
\end{enumerate}
In Case (1) and when $d\ge 1$, equality is achieved only when $C$ has the following classes:
\begin{itemize}
    \item When $d>1$, $q^*\Gamma$;
    \item When $d=1$, $q^*\Gamma, q^*(\Delta_0+\Gamma), \widetilde{\Delta}_0+\widetilde{\Gamma}+E_x$ or $\widetilde{\Delta}_0+\widetilde{\Gamma}+E_y.$
\end{itemize}
}
\end{proposition}

\begin{proof}
We use the same argument as Proposition~\ref{prop:hirzebruch}.
At a general $b\in B$, we have:
\[\deg N_{f_b/X}(\log D_{f,b}
)=2g-2+i-((a-2)q^*\Delta_0 + (b-d-1)q^*\Gamma -c_xE_x -c_yE_y)\cdot C.\]
Also note that $T_W(-\log \widetilde{\Gamma})$ is pseudo-nef outside of $E_x\cup E_y\cup\widetilde{\Delta}_0$ (Lemma~\ref{prop:nefness-2points-d>0}).
Therefore, Proposition~\ref{prop:log-normal-degree-bound} and Lemma~\ref{lemma:secdom-hirzebruch2} imply that we have the following cases:

\textbf{Case 1:} 
There is a generically surjective map 
$$\mu: f_b^*M_{q^*(\Delta_0+d\Gamma)}\to N_{f_b/W}(\log D_{f_b}).$$
Let us denote the image sheaf of $\mu$ by $Q$.
Arguing as in Proposition~\ref{prop:hirzebruch-horikawa}, we have $\deg Q\ge -q^*(\Delta_0+d\Gamma)\cdot C_b$, which implies (1). 

If equality occurs, then $f_b^*M_{q^*(\Delta_0+d\Gamma)}$ must possess at least $d$ global sections. 
By Proposition~\ref{prop:sections-of-kernel-bundles}, if the image of $C_b$ is an integral curve of class $\alpha \widetilde{\Delta}_0+\beta \widetilde{\Gamma}+\gamma_xE_x+\gamma_yE_y$, we have:
$$
\begin{aligned}
h^0(C_b, f_b^*M_{q^*(\Delta_0+d\Gamma)}) &= h^0(W, \mathcal{I}_{f_b(C_b)/W}(q^*(\Delta_0+d\Gamma))) \\
&= h^0(W, \mathcal{O}_{W}((1-\alpha)\widetilde{\Delta}_0 + (d-\beta)\widetilde{\Gamma}+(d-\gamma_x)E_x+(d-\gamma_y)E_y)).
\end{aligned}
$$
Since $f_b(C_b) \not\subset \widetilde{\Delta}_0\cup \widetilde{\Gamma}\cup E_x \cup E_y$, we must have $\alpha+\gamma_x+\gamma_y\ge 2\beta$, $\beta \ge \alpha d$, and $\beta\ge \gamma_x,\gamma_y$.
For this dimension to be at least $d$, we are limited to two possibilities:
\begin{itemize}
    \item $\alpha=0$. This forces $\beta=\gamma_x=\gamma_y=1$ for the curve to remain integral and not contained in the boundary components, corresponding to the class $q^*\Gamma$. This yields $h^0(W, \mathcal{O}_{W}(q^*\Delta_0 + (d-1)q^*\Gamma)) = d$.
    \item $\alpha=1$. This forces $d=1$ and the curve to be in the class of $q^*(\Delta_0+\Gamma)$, $q^*(\Delta_0+\Gamma)-E_x$, or $q^*(\Delta_0+\Gamma)-E_y$. 
\end{itemize}

\textbf{Case 2:} There is a generically surjective map 
$$\mu: f_b^*M_{q^*\Gamma}\to N_{f_b/W}(\log D_{f_b}),$$ so the second inequality follows.
\end{proof}

\begin{proposition}\label{prop:hirzebruch2-f0}
{\em
Let $q:W\to \ff_0$ be the blowup of $\ff_0$ at two distinct points $x,y$ lying on the same fiber $\Gamma_0$.
Let $D'\subset W$ be a very general curve in the linear system $$|a q^*\Delta_0 + b q^*\Gamma - c_x E_x - c_y E_y|$$ with $c_x,c_y\ge 0$, $a\ge c_x+c_y$ and $b\ge \max\{c_x,c_y\}$, and let $D=\widetilde{\Gamma}+D'.$
Let $f:C\to W$ be a map from a smooth projective curve $C$ of genus $g$ that is birational onto its image, satisfies $f(C)\not\subset D\cup E_x\cup E_y\cup\widetilde{\Delta}_x\cup \widetilde{\Delta}_y$,
and meets $D$ in exactly $i$ distinct points. 
Then
\begin{enumerate}
    \item either $2g-2+i\ge ((a-3) q^*\Delta_0 + (b-1) q^*\Gamma-c_x E_x-c_yE_y)\cdot C$,
    \item or $2g-2+i\ge ((a-2) q^*\Delta_0 + (b-2) q^*\Gamma-c_x E_x-c_yE_y)\cdot C.$
\end{enumerate}
}
\end{proposition}

\begin{remark}
Propositions~\ref{prop:hirzebruch2} and \ref{prop:hirzebruch2-f0} imply that when $D'\subset W$ is a very general curve in  $|a q^*\Delta_0 + b q^*\Gamma - c_x E_x - c_y E_y|$ with $c_x,c_y\ge 2$, $a\ge c_x+c_y+4$ and $b\ge \begin{cases}
   ad+\max\{c_x,c_y\} & d>0\\
   3+\max\{c_x,c_y\} & d=0,
\end{cases}$
the pair $(W,\widetilde{\Gamma}+D')$ is log algebraically hyperbolic.
\end{remark}

Now, we apply these results to study double covers of $W$ that give rise to Horikawa surfaces. In particular, we are interested in the case where the branch locus $D$ is the disjoint union of $\widetilde{\Gamma}$ and a very general divisor $D'$ in the linear system
$$|6 q^*\Delta_0 + (n+4+3d) q^*\Gamma - 3 E_x - 3 E_y|,$$
with $n-d$ odd and $d\le\frac{n+4}{3}$.
In order to apply our previous section-dominance bounds to $D'$, the coefficients of this linear system must satisfy $d\le \frac{n+1}{3}$. 
Due to this, we miss the stratum in the moduli where $d=\frac{n+3}{3}$. Consequently, the results in this subsection will specifically apply to the strata where $d\le \frac{n+1}{3}$.

\begin{proposition}
{\em
Let $d> 0$ and let $q:W\to \ff_d$ be the blowup of $\ff_d$ at two distinct points $x,y$ lying on the same fiber $\Gamma_0$, and neither on $\Delta_0$. 
Let $D'\subset W$ be a very general curve in the linear system $$|6 q^*\Delta_0 + (n+4+3d) q^*\Gamma - 3 E_x - 3 E_y|$$ with 
\begin{itemize}
    \item either $d\ge 2$ and $n \ge 3d-1$, 
    \item or $d=1$ and $n\ge 4$.
\end{itemize}
Let $D=\widetilde{\Gamma}+D'.$
For any map $f:C\to W$ from a smooth projective curve $C$ of genus $g$ that is birational onto its image, satisfies $f(C)\not\subset D\cup E_x\cup E_y\cup \widetilde{\Delta}_0$, and meets $D$ in exactly $i$ distinct points, we have:
\begin{equation*}
2g-2+i\ge \frac{1}{2}D\cdot C.    
\end{equation*}
Moreover, equality is only possible when $f(C)$ is of class $q^*\Gamma$.
}    
\end{proposition}

\begin{proof}
Let $C$ have class $\alpha \widetilde{\Delta}_0+ \beta \widetilde{\Gamma} + \gamma_x E_x+ \gamma_y E_y$.
Applying the bounds from Proposition~\ref{prop:hirzebruch2}, we deduce that when $n\ge 3d-1$, we have two cases:

\textbf{Case 1:}
$$2g-2+i-\frac{1}{2}D\cdot C \ge \left(\frac{n+1-d}{2} q^*\Gamma - E_x-E_y\right)\cdot C = \frac{n+1-d}{2}\alpha - 2\beta+\gamma_x+\gamma_y.$$

The right-hand side is always non-negative, and it is equal to zero only when $f(C)$ has class $q^*\Gamma.$

\textbf{Case 2:} 
$$2g-2+i-\frac{1}{2}D\cdot C \ge \left(q^*\Delta_0+\frac{n-1+d}{2} q^*\Gamma - E_x-E_y\right)\cdot C = \frac{n-1-d}{2}\alpha - \beta+\gamma_x+\gamma_y.$$

The right-hand side is strictly positive.

\end{proof}

\begin{proposition}
{\em
Let $q:W\to \ff_0$ be the blowup of $\ff_0\simeq\pp^1\times \pp^1$ at two distinct points $x,y$ lying on the same fiber $\Gamma_0$.
Let $D'\subset W$ be a very general curve in the linear system $$|6 q^*\Delta_0 + (n+4) q^*\Gamma - 3 E_x - 3 E_y|$$ with $n \ge 5$, and let $D=\widetilde{\Gamma}+D'.$
For any map $f:C\to W$ from a smooth projective curve $C$ of genus $g$ that is birational onto its image, satisfies $f(C)\not\subset D\cup E_x\cup E_y\cup \widetilde{\Delta}_x\cup \widetilde{\Delta}_y$, and meets $D$ in exactly $i$ distinct points, we have:
\begin{equation*}
2g-2+i \ge \frac{1}{2}D\cdot C.    
\end{equation*}
Moreover, equality is only possible when $f(C)$ is of class $q^*\Gamma$.
}    
\end{proposition}

\begin{proof}
Let $C$ have class $\alpha\widetilde{\Gamma}+\beta_x\widetilde{\Delta}_x+\beta_y\widetilde{\Delta}_y+\gamma_xE_x+\gamma_yE_y.$

\textbf{Case 1:}
$$2g-2+i-\frac{1}{2}D\cdot C \ge (\frac{n+1}{2}\cdot q^*\Gamma- E_x-E_y)\cdot C=\frac{n-1}{2}(\beta_x+\beta_y) + \gamma_x+\gamma_y-2\alpha.$$

The right-hand side is always non-negative and it is equal to zero only when $f(C)$ has class $q^*\Gamma$.

\textbf{Case 2:} $$2g-2+i-\frac{1}{2}D\cdot C \ge (q^*\Delta_0+\frac{n-1}{2}\cdot q^*\Gamma- E_x-E_y)\cdot C=\frac{n-3}{2}(\beta_x+\beta_y) + \gamma_x+\gamma_y-\alpha.$$

Since $n\ge 5$, the right hand side is always positive.
\end{proof}

\begin{corollary}\label{cor:hirzebruch-2pt-2comp-horikawa-cover}
{\em
Let $d>0$ and let $\rho:S'\to W$ be a double cover branched over $D=\widetilde{\Gamma}+D'$ where $D'$ is a very general curve in the linear system $|6 q^*\Delta_0 + (n+4+3d)  q^*\Gamma - 3  E_x - 3 E_y|$ satisfying 
\begin{itemize}
    \item either $d\ge 2$ and $n \ge 3d-1$, 
    \item or $d=1$ and $n\ge 4$.
\end{itemize}
Then, $\rho^{-1}(\widetilde{\Gamma})$ is the only rational curve on $S'.$
Furthermore, if $C'\subset S'$ is an integral curve with $g(C')=1$ not mapping to $E_x\cup E_y\cup \widetilde{\Delta}_0$, then $C'=\rho^{-1}(C)$ for an integral curve $C$ of class $q^*\Gamma$ that is tangent to $D'$ at a point outside $\operatorname{Sing}(D).$ 
In particular, $S'$ is pseudo-Lang algebraically hyperbolic.
}
\end{corollary}

\begin{remark}
Consider the curves mapping to $E_x\cup E_y\cup \widetilde{\Delta}_0$ in Corollary~\ref{cor:hirzebruch-2pt-2comp-horikawa-cover}.
\begin{itemize}
    \item $\rho^{-1}(E_x)$ and $\rho^{-1}(E_y)$ are elliptic curves since $E_x$ and $E_y$ intersect $D$ transversally at 4 points.
    \item $\widetilde{\Delta}_0$ also lifts to an elliptic curve when $n=3d-1$.
    This is because $\widetilde{\Delta}_0$ intersects $D$ transversally at $n+5-3d$ points, which yields a genus of $g=1$ via the Riemann-Hurwitz formula. When $n\ge 3d$, $\widetilde{\Delta}_0$ lifts to a curve of higher genus.
\end{itemize}
\end{remark}

\begin{corollary}\label{cor:hirzebruch-f0-2pt-2comp-horikawa-cover}
{\em
Let $d=0$ and let $\rho:S'\to W$ be a double cover branched over $D=\widetilde{\Gamma}+D'$ where $D'$ is a very general curve in the linear system $$|6 q^*\Delta_0 + (n+4) q^*\Gamma - 3 E_x - 3 E_y|$$ with $n \ge 5$.
Then, $\rho^{-1}(\widetilde{\Gamma})$ is the only rational curve on $S'.$
Furthermore, if $C'\subset S'$ is an integral curve with $g(C')=1$ not mapping to $E_x\cup E_y\cup \widetilde{\Delta}_x\cup \widetilde{\Delta}_y$, then $C'=\rho^{-1}(C)$ for an integral curve $C$ of class $q^*\Gamma$ that is tangent to $D'$ at a point outside $\operatorname{Sing}(D).$ 
In particular, $S'$ is pseudo-Lang algebraically hyperbolic.
}
\end{corollary}

\begin{remark}
Consider the preimages of the exceptional curves in Corollary~\ref{cor:hirzebruch-f0-2pt-2comp-horikawa-cover}.
\begin{itemize}
    \item $\rho^{-1}(E_x)$ and $\rho^{-1}(E_y)$ are elliptic curves since $E_x$ and $E_y$ intersect $D$ transversally at 4 points.
    \item $\widetilde{\Delta}_x$ and $\widetilde{\Delta}_y$ lift to curves of higher genus, since each intersects $D$ transversally at $n+1$ points.
\end{itemize}
\end{remark}

We shall summarize the characterization and count of rational and elliptic curves on a very general Horikawa surface of the second kind in each stratum in \S\ref{sec:summary-2nd} and Table~\ref{table:horikawa-count-2}.

\subsection{Blowup of Hirzebruch surfaces at two points on the same fiber, only one on $\Delta_0$}
\label{sec:hirzebruch-2pt-3comp}

In this subsection, we assume $d>0$ and denote by $q\colon W\to \ff_d$ the blowup of $\ff_d$ at two distinct points $x$ and $y$ lying on the same fiber $\Gamma_0$, where $x\in \Delta_0$ and $y\notin \Delta_0$.

For a general point in each stratum of Horikawa surfaces of the second kind, specifically those of type $(d)$ for $n\ge 4$ with $d>\frac{n+4}{3}$, the corresponding surface is birational to a double cover of $W$, where $y$ is a generic point on $\Gamma_0$. 
The branch locus is the disjoint union of $\widetilde{\Gamma}_0$ and $\widetilde{\Delta}_0\cup D'$, where $D'$ is a general divisor. 
Moreover, when $d=\frac{n+2}{2}$ (which only occurs when $4 \mid n$ and $n \ge 8$), such Horikawa surfaces form a separate component of the moduli space, and the branch locus is a mutually disjoint union of three components: $\widetilde{\Gamma}$, $\widetilde{\Delta}_0$, and $D'$. 
Surfaces of type $(3')$ also fall into this configuration.

First, we establish the relevant positivity result for the log tangent bundle $T_W(-\log (\widetilde{\Gamma}+\widetilde{\Delta}_0))$.

\begin{proposition}\label{prop:nefness-2points-d>0-xinD}
{\em
Let $d>0$ and let $q:W\to \ff_d$ the blowup of $\ff_d$ at two distinct points $x$ and $y$ lying on the same fiber $\Gamma_0$, where $x\in \Delta_0$ and $y\notin \Delta_0.$  
Assume that there is a non-constant map $f:C\to W$ from a smooth projective curve and a surjective map 
$$ f^*T_W(-\log( \widetilde{\Gamma}+\widetilde{\Delta}_0))\to L,$$
where $L$ is a line bundle with $\deg L<0$. Then $f(C)$ is either $E_x$ or $E_y$.
}
\end{proposition}

\begin{proof}
The divisor $D=\widetilde{\Gamma}+\widetilde{\Delta}_0+E_x+E_y$ is normal crossing, so there is a short exact sequence
$$ \xymatrix{
 0\ar[r] &    T_W(-\log D)  \ar[r] &  T_W(-\log (\widetilde{\Gamma}+\widetilde{\Delta}_0)) \ar[r] & \oo_{E_x}(E_x)\oplus  \oo_{E_y}(E_y) \ar[r] & 0.
} $$
As in the proof of Proposition \ref{prop:nefness-2points-d=0}, it is enough to prove that $T_W(-\log D)$ is a nef vector bundle. 
Let $p:W\to \pp^1$ be the composition of the blowup $q:W\to \ff_d$ with the fibration $\ff_d\to \pp^1$.  
By the same argument as in Proposition \ref{prop:nefness-2points-d>0}, there is a morphism of vector bundles $\varphi:T_W(-\log D)\to p^*T_{\pp^1}$ whose image is $p^*\oo_{\pp^1}(1)$.
We have the following short exact sequence
$$ \xymatrix{
0\ar[r] & K \ar[r] &  T_W(-\log D) \ar[r]  &  p^*\oo_{\pp^1}(1) \ar[r] & 0
} $$ 
where the sheaf $K$ is a line bundle isomorphic to $\oo_W(q^*(\Delta_0+d\Gamma)-E_y)$, which is nef. 
Indeed, we may directly verify that it intersects $\widetilde{\Gamma},\widetilde{\Delta}_0,E_x$ and $E_y$ non-negatively. 
If $f(C) \notin \{E_x, E_y, \widetilde{\Delta}_0, \widetilde{\Gamma}\}$ and it has class $a\widetilde{\Delta}_0+b\widetilde{\Gamma}+c_xE_x+c_yE_y$, then $K\cdot C = c_y\ge0$.
Therefore, $T_W(-\log D)$ is an extension of two nef line bundles.  
\end{proof}

Next, we state the corresponding vanishing and section-dominance results. 
We omit some of the proofs since they are the same as those in \S\ref{sec:hirzebruch-2pt}.

\begin{lemma}\label{lemma:h1-vanishing-hirzebruch2-xindelta_0}
{\em
Let $d>0$ and let $q:W\to \ff_d$ the blowup of $\ff_d$ at two distinct points $x$ and $y$ lying on the same fiber $\Gamma_0$, where $x\in \Delta_0$ and $y\notin \Delta_0.$   
Let $$L=\oo_W(a\cdot q^*\Delta_0 + b \cdot q^* \Gamma - c_x\cdot E_x-c_y\cdot E_y)$$ with $a,b,c_x,c_y\ge 0$. If $a\ge c_x+c_y-1$, $b\ge ad+c_x-1$ and $c_x\ge1$, then $H^1(L)=0.$
}
\end{lemma}

\begin{proof}
First, we show by an induction argument that it suffices to show that $H^1$ vanishes for the line bundle of $(a-c_x)\cdot q^*\Delta_0 + b \cdot q^* \Gamma - c_y\cdot E_y$.
For $0 \le k \le c_x,$ we denote by $L_k$ the line bundle of $(a-k)\cdot q^*\Delta_0 + b \cdot q^* \Gamma - (c_x-k)\cdot E_x-c_y\cdot E_y$.
For each $0 \le k \le c_x-1$, restricting $L_k$ to $\widetilde{\Delta}_0$ gives the short exact sequence 
\begin{equation*}
    0 \to L_{k+1} \to L_k \to L_k\vert_{\widetilde{\Delta}_0}\to 0.
\end{equation*}
Since $L_k\vert_{\widetilde{\Delta}_0}\simeq \oo_{\pp^1}(b-(a-k)d-c_x+k),$ $H^1(\widetilde{\Delta}_0, L_k\vert_{\widetilde{\Delta}_0})=0$ when $b-(a-k)d-c_x+k\ge -1$. 
This holds under our assumption that $b\ge ad + c_x-1.$
Since $H^1(W, L_{k+1})\twoheadrightarrow H^1(W,L_k)$, by induction it is enough to show that $H^1(W, L_{c_x})=0.$

Now, a similar induction argument shows that it suffices to show that $H^1$ vanishes for the line bundle of $(a-c_x-c_y)\cdot q^*\Delta_0 + (b-c_y\cdot d) \cdot q^* \Gamma$. 
For $0\le k \le c_y$, we denote by $M_k$ the line bundle of $(a-c_x-k)\cdot q^*\Delta_0 + (b-kd) \cdot q^* \Gamma - (c_y-k)\cdot E_y$.
Let $C_y$ denote a smooth rational curve in $|q^*(\Delta_0+d\Gamma)-E_y|$, in other words the strict transform of a smooth rational curve of class $\Delta_0+d\Gamma$ on $\ff_d$ that passes through $y.$
For each $0 \le k \le c_y-1$, restricting $M_k$ to $C_y$ gives 
\begin{equation*}
    0 \to M_{k+1} \to M_k \to M_k\vert_{C_y}\to 0.
\end{equation*}
Since $M_k\vert_{C_y}\simeq \oo_{\pp^1}(b-k(d-1)-c_y)$, it follows that $H^1(C_y,M_k\vert_{C_y})=0$ when $b-k(d-1)-c_y\ge-1$, which holds under the inequalities $b\ge ad + c_x-1$ and $c_x\ge 1.$ 
Hence we have a surjection $H^1(W, M_{k+1})\twoheadrightarrow H^1(W, M_k)$, and our claim holds by induction.

Finally, we have $$H^1(W, \oo_W((a-c_x-c_y)\cdot q^*\Delta_0 + (b-c_y\cdot d) \cdot q^* \Gamma))\simeq H^1(\ff_d, \oo_{\ff_d}((a-c_x-c_y)\cdot \Delta_0 + (b-c_y\cdot d)\cdot \Gamma)).$$ 
This cohomology group vanishes when $a\ge c_x+c_y-1$ and $b\ge (a-c_x)d-1$, where the latter follows from $b\ge ad+c_x-1.$
\end{proof}

\begin{lemma}
{\em
Let $d>0$ and let $q:W\to \ff_d$ the blowup of $\ff_d$ at two distinct points $x$ and $y$ lying on the same fiber $\Gamma_0$, where $x\in \Delta_0$ and $y\notin \Delta_0.$   
Let $$L=\oo_W(a\cdot q^*\Delta_0 + b \cdot q^* \Gamma - c_x\cdot E_x-c_y\cdot E_y)$$ be a line bundle with $c_x\ge1$, $c_y\ge 0$, $a\ge c_x+c_y$ and $b\ge ad+c_x$.
At every point $p\in W\setminus q^{-1}(\Gamma_0)$, there is a surjective map
\begin{equation*}
\begin{aligned}
&\left( H^0(W, \mathcal{O}_{W}(q^*(\Delta_0+d\Gamma)) \otimes \mathcal{I}_{p/W}) \otimes H^0(W, \mathcal{O}_{W}((a-1)q^*\Delta_0 + (b-d)q^*\Gamma-c_xE_x-c_yE_y) \right) \\
&\oplus \left( H^0(W, \mathcal{O}_{W}(q^*\Gamma) \otimes \mathcal{I}_{p /W}) \otimes H^0(W, \mathcal{O}_{W}(aq^*\Delta_0 + (b-1)q^*\Gamma-c_xE_x-c_yE_y) \right) \\
&\to H^0(W,L \otimes \mathcal{I}_{p /W}).
\end{aligned}
\end{equation*}
In other words, outside of $ q^{-1}(\Gamma_0)$, there is a surjective map
\begin{equation*}
M_{q^*(\Delta_0+d\Gamma)}^{\oplus s}\oplus M_{q^*\Gamma}^{\oplus t}\to M_L.
\end{equation*}
}
\end{lemma}

\begin{proposition}
{\em
Let $d>0$ and let $q:W\to \ff_d$ the blowup of $\ff_d$ at two distinct points $x$ and $y$ lying on the same fiber $\Gamma_0$, where $x\in \Delta_0$ and $y\notin \Delta_0.$ 
Let $D'\subset W$ be a very general curve in the linear system $$|a\cdot q^*\Delta_0 + b\cdot q^*\Gamma - c_x\cdot E_x - c_y\cdot E_y|$$ with $c_x\ge1$, $c_y\ge 0$, $a\ge c_x+c_y$ and $b\ge ad+c_x$, and let $D=\widetilde{\Gamma}+\widetilde{\Delta}_0+D'.$
Let $f:C\to W$ be a map from a smooth projective curve $C$ of genus $g$ that is birational onto its image, satisfies $f(C)\not\subset D\cup E_x\cup E_y$, 
and meets $D$ in exactly $i$ distinct points. 
Then
\begin{enumerate}
    \item either $2g-2+i\ge ((a-2) q^*\Delta_0 + (b-2d-1) q^*\Gamma-(c_x+1) E_x-c_yE_y)\cdot C$,
    \item or $2g-2+i\ge ((a-1) q^*\Delta_0 + (b-d-2) q^*\Gamma-(c_x+1) E_x-c_yE_y)\cdot C.$
\end{enumerate}
In Case (1), equality can occur only if $f(C)$ has the following classes:
\begin{itemize}
    \item When $d>1$, $q^*\Gamma$;
    \item When $d=1$, $q^*\Gamma$, $q^*(\Delta_0+\Gamma)$, $q^*(\Delta_0+\Gamma)-E_x$, or $q^*(\Delta_0+\Gamma)-E_y.$
\end{itemize}
}
\end{proposition}

\begin{remark}
    The above proposition implies that when $D'\subset W$ is a very general curve in $|a\cdot q^*\Delta_0 + b\cdot q^*\Gamma - c_x\cdot E_x - c_y\cdot E_y|$ with $c_x\ge 1$, $c_y\ge 2$, $a\ge c_x+c_y+4$ and $b\ge ad+c_x+4$, then $(W,\widetilde{\Gamma}+\widetilde{\Delta}_0+D')$ is log algebraically hyperbolic.
\end{remark}

We now apply the previous results to the linear systems for $D'$ that produce Horikawa surfaces.

\begin{proposition}
{\em
Let $d>0$ and let $q:W\to \ff_d$ the blowup of $\ff_d$ at two distinct points $x$ and $y$ lying on the same fiber $\Gamma_0$, where $x\in \Delta_0$ and $y\notin \Delta_0.$ 
Let $D'\subset W$ be a very general curve in the linear system $$|5 q^*\Delta_0 + (n+4+3d) q^*\Gamma - 2 E_x - 3 E_y|$$ satisfying 
\begin{itemize}
    \item either $d\ge4$ and $n\ge 2d-2$, 
    \item or $d\le 3$ and $n\ge d+2$.
\end{itemize}
Let $$D=\widetilde{\Gamma}+\widetilde{\Delta}_0+D'.$$
For any map $f:C\to W$ from a smooth projective curve $C$ of genus $g$ that is birational onto its image, satisfies $f(C)\not\subset D\cup E_x\cup E_y$, and meets $D$ in exactly $i$ distinct points, we have:
\begin{equation*}
2g-2+i\ge \frac{1}{2}D\cdot C.    
\end{equation*}
Moreover, equality is only possible when $f(C)$ is of class $q^*(\Gamma).$
}    
\end{proposition}

\begin{proof}
Let $C$ have class $\alpha \widetilde{\Delta}_0+ \beta \widetilde{\Gamma}+\gamma_x E_x+ \gamma_y E_y$.
Then, we have two cases:

\textbf{Case 1:}
$$2g-2+i-\frac{1}{2}D\cdot C \ge \left(\frac{n-d+1}{2} q^*\Gamma- E_x-E_y\right)\cdot C = \frac{n-d-1}{2}\alpha - 2\beta + \gamma_x + \gamma_y.$$
The right-hand side is always non-negative and it is equal to zero only when $f(C)$ has class $q^*\Gamma$.

\textbf{Case 2:} 
$$2g-2+i-\frac{1}{2}D\cdot C \ge \left(q^*\Delta_0+\frac{n+d-1}{2} q^*\Gamma- E_x-E_y\right)\cdot C = \frac{n-d-3}{2}\alpha-\beta + \gamma_x+ \gamma_y.$$
The right-hand side is always positive.

\end{proof}

\begin{corollary}\label{cor:hirzebruch-2pt-3comp-horikawa-cover}
{\em
Let $d>0$ and let $\rho\colon S' \to W$ be a double cover branched over $D=\widetilde{\Gamma}+\widetilde{\Delta}_0+D'$ where $D'$ is a very general curve in the linear system $$|5\cdot q^*\Delta_0 + (n+4+3d)\cdot q^*\Gamma- 2\cdot E_x - 3\cdot E_y|$$ satisfying 
\begin{itemize}
    \item either $d\ge4$ and $n\ge 2d-2$, 
    \item or $d\le 3$ and $n\ge d+2$.
\end{itemize}
Then, $\rho^{-1}(\widetilde{\Gamma})$ and $\rho^{-1}(\widetilde{\Delta}_0)$ are the only rational curves on $S'$.
Furthermore, if $C'\subset S'$ is an integral curve with $g(C')=1$ not mapping to $E_x\cup E_y$, then $C'=\rho^{-1}(C)$ for an integral curve $C$ of class $q^*\Gamma$ that is either tangent to $D'$ or intersects $D$ at a node.
In particular, $S'$ is pseudo-Lang algebraically hyperbolic.
}
\end{corollary}

\begin{remark}
    In Corollary~\ref{cor:hirzebruch-2pt-3comp-horikawa-cover}, $\rho^{-1}(E_x)$ and $\rho^{-1}(E_y)$ are elliptic curves in $S'$ since $E_x$ and $E_y$ intersect $D$ transversally at 4 points.
\end{remark}

\subsection{Blowup of Hirzebruch surfaces at one point on $\Delta_0$}\label{sec:hirzebruch-1pt}
In this final subsection, we assume $d>0$ and denote by $q\colon W\to \ff_d$ the blowup of $\ff_d$ at a point $x\in \Delta_0$.
We denote by $\Gamma_0$ the curve of class $\Gamma$ through $x$, $\widetilde{\Delta}_0$ the proper transform of $\Delta_0$ in $X$, and $E$ the exceptional divisor of the blow-up. 
For $n=4$, the moduli space of Horikawa surfaces of the second kind consists of two components. 
A general point in one of these components corresponds to a surface of type $(1^*)$, which is birational to a double cover of $W$ branched over the disjoint union of $\widetilde{\Delta}_0$ and a general divisor.

We omit the proofs of the following results, as they are similar to those in \S\S\ref{sec:hirzebruch-2pt}-\ref{sec:hirzebruch-2pt-3comp}.

\begin{proposition}
{\em
Let $d>0$ and let $q\colon W\to \ff_d$ be the blowup of $\ff_d$ at a point $x\in \Delta_0.$
Assume that there is a non-constant map $f:C\to W$ from a smooth projective curve and a surjective morphism $$f^*T_W(-\log \widetilde{\Delta}_0)\to  L,$$ where $L$ is a line bundle with $\deg L<0$. Then $f(C)= E_x$.
}
\end{proposition}

\begin{lemma}
{\em
Let $d>0$ and let $q\colon W\to \ff_d$ be the blowup of $\ff_d$ at a point $x\in \Delta_0.$ 
Consider the line bundle $$L=\oo_W(a\cdot q^*\Delta_0 + b \cdot q^* \Gamma - c_x\cdot E_x)$$ with $a,b,c_x\ge 0$. If $a\ge c_x-1$ and $b\ge ad+c_x-1$, then $H^1(L)=0.$
}
\end{lemma}

\begin{lemma}
{\em
Let $d>0$ and let $q\colon W\to \ff_d$ be the blowup of $\ff_d$ at a point $x\in \Delta_0.$ 
Let $$L=\oo_W(a\cdot q^*\Delta_0 + b \cdot q^* \Gamma - c_x\cdot E_x)$$ with $c_x\ge 0$, $a\ge c_x$ and $b\ge ad+c_x$. 
Similar to Lemma~\ref{lemma:secdom-hirzebruch2}, outside of $q^{-1}(\Gamma_0)$, there is a surjective map
\begin{equation*}
M_{q^*(\Delta_0+d\Gamma)}^{\oplus s}\oplus M_{q^*\Gamma}^{\oplus t}\to M_L.
\end{equation*}
}
\end{lemma}

\begin{proposition}\label{prop:hirzebruch-1pt}
{\em
Let $d>0$ and let $q\colon W\to \ff_d$ be the blowup of $\ff_d$ at a point $x\in \Delta_0.$ 
Let $D'\subset W$ be a very general curve in the linear system $$|a\cdot q^*\Delta_0 + b\cdot q^*\Gamma-c_x\cdot E_x|$$ where $c_x\ge0$, $a\ge c_x$ and $b\ge ad+c_x$, and let $D=\widetilde{\Delta}_0+D'.$
Let $f:C\to W$ be a map from a smooth projective curve $C$ of genus $g$ that is birational onto its image, satisfies $f(C)\not\subset D\cup \widetilde{\Gamma}\cup E_x$, and meets $D$ in exactly $i$ distinct points. Then 
\begin{enumerate}
    \item either $2g-2+i\ge ((a-2) q^*\Delta_0 + (b-2d-2) q^*\Gamma-c_x E_x)\cdot C$,
    \item or $2g-2+i\ge ((a-1) q^*\Delta_0 + (b-d-3) q^*\Gamma-c_xE_x)\cdot C.$
\end{enumerate}
In Case (1), equality is achieved only when $f(C)$ has the following classes:
\begin{itemize}
    \item When $d>1$, $q^*\Gamma$;
    \item When $d=1$, $q^*\Gamma$ or $q^*(\Delta_0+\Gamma)$.
\end{itemize}
}
\end{proposition}

\begin{remark}
    The above proposition implies that when $D'\subset W$ is a very general curve in $|a\cdot q^*\Delta_0 + b\cdot q^*\Gamma - c_x\cdot E_x|$ with $c_x\ge 2$, $a\ge c_x+3$ and $b\ge ad+c_x+4$, then $(W,\widetilde{\Delta}_0+D')$ is log algebraically hyperbolic.
\end{remark}

Surfaces of type $(1^*)$ are double covers of $W$ with $d=1$ branched over the disjoint union of $\widetilde{\Delta}_0$ and a very general divisor in the linear system $|7 q^*\Delta_0 + 10 q^*\Gamma-3 E_x|$.

\begin{proposition}\label{prop:hirzebruch-1pt-horikawa}
{\em
Let $q\colon W\to \ff_1$ be the blowup of $\ff_1$ at a point $x\in \Delta_0.$ 
Let $D'\subset W$ be a very general curve in the linear system $|7 q^*\Delta_0 + 10 q^*\Gamma-3 E_x|$,
and let $D=\widetilde{\Delta}_0+D'.$ 
For any map $f:C\to W$ from a smooth projective curve $C$ of genus $g$ that is birational onto its image, satisfies $f(C)\not\subset D\cup \widetilde{\Gamma}\cup E_x$, and meets $D$ in exactly $i$ distinct points, we have
\begin{equation*}
2g-2+i>\frac{1}{2}D\cdot C.    
\end{equation*}

}    
\end{proposition}

\begin{proof}
As before, we have two cases:

\textbf{Case 1:} We have $2g-2+i\ge(5 q^*\Delta_0 + 6 q^*\Gamma-3 E_x)\cdot C$, so
\[2g-2+i-\frac{1}{2}D\cdot C\ge( q^*\Delta_0 + q^*\Gamma- E_x)\cdot C.\]

\textbf{Case 2:} $2g-2+i\ge(6 q^*\Delta_0 + 6 q^*\Gamma-3 E_x)\cdot C$, so
\[2g-2+i-\frac{1}{2}D\cdot C\ge(2 q^*\Delta_0 + q^*\Gamma- E_x)\cdot C.\]
In both cases, the right-hand side is always positive.
\end{proof}

\begin{corollary}\label{cor:hirzebruch-1pt-horikawa-cover}
{\em
Let $d=1$ and let $\rho\colon S' \to W$ be a double cover branched over $D=\widetilde{\Delta}_0+D'$ where $D'$ is a very general curve in the linear system $|7 q^*\Delta_0 + 10 q^*\Gamma-3 E_x|$.
Then, the only rational curve in $S'$ is $\rho^{-1}(\widetilde{\Delta}_0)$. 
Furthermore, every integral curve $C'\subset S'$ with $g(C')=1$ maps to $\widetilde{\Gamma}\cup E_x$.
In particular, $S'$ is pseudo-Lang algebraically hyperbolic.
}
\end{corollary}

\begin{remark}
    In Corollary~\ref{cor:hirzebruch-1pt-horikawa-cover}, $\rho^{-1}(\widetilde{\Gamma})$ and $\rho^{-1}(E_x)$ are elliptic curves in $S'$ since they each intersect $D$ transversally at 4 points.
\end{remark}

\subsection{Summary for very general Horikawa surfaces of the second kind}\label{sec:summary-2nd}

We have shown that a very general Horikawa surface of the second kind with $n\ge 4$ is pseudo-Lang algebraically hyperbolic. 
Specifically, the distribution of rational or elliptic curves across the moduli space is as follows:
\begin{itemize}
\item \textbf{When $n=4$:} 
    The moduli space has two components. 
    \begin{itemize}
        \item For the first component, a very general surface (type $(1)$) contains no rational curves and $(10n+28)+2=10n+30=70$ elliptic curves ($E_x$, $E_y$, and curves in $|q^*\Gamma|$ tangent to $D'$) by Corollary~\ref{cor:hirzebruch-2pt-2comp-horikawa-cover}. 
        \item For the second component, a very general surface in the big stratum (type $(1^*)$) contains no rational curves and $2$ elliptic curves ($\widetilde{\Gamma}$ and $E_x$) by Corollary~\ref{cor:hirzebruch-1pt-horikawa-cover}. 
        Our result does not apply to the small stratum (type $(3')$) (see Corollary~\ref{cor:hirzebruch-2pt-3comp-horikawa-cover}).
    \end{itemize}

\item \textbf{When $n=5$:} 
The moduli space has two components. 
\begin{itemize}
    \item For the first component, a very general surface in the big stratum (type $(0)$) contains no rational curves and $(10n+28)+2=10n+30=80$  elliptic curves ($E_x$, $E_y$, and curves in $|q^*\Gamma|$ tangent to $D'$) by Corollary~\ref{cor:hirzebruch-f0-2pt-2comp-horikawa-cover}. 
    A very general surface in the small stratum (type $(2)$) contains no rational curves and $(10n+28)+3=10n+31=81$ elliptic curves ($E_x$, $E_y$, $\widetilde{\Delta}_0$ and curves in $|q^*\Gamma|$ tangent to $D'$) by Corollary~\ref{cor:hirzebruch-2pt-2comp-horikawa-cover}.  
    \item For the second component, a very general surface (type $(2^*)$) contains no rational or elliptic curves (Corollary~\ref{cor:hirzebruch-reducible-horikawa-cover}).
\end{itemize}

\item \textbf{When $4\mid n$ and $n\ge 8$:} The moduli space has two components. 
\begin{itemize}
    \item For the first component, a very general surface in the larger strata (types $(1)$ down to $(d)$ with $d\le \frac{n+4}{3}$) contains: 
    \begin{itemize}
        \item No rational curves and $(10n+28)+2=10n+30$ elliptic curves ($E_x$, $E_y$ and curves in $|q^*\Gamma|$ tangent to $D'$) when $d<\frac{n+1}{3}$. 
        \item No rational curves and $(10n+28)+3=10n+31$ elliptic curves ($E_x$, $E_y$, $\widetilde{\Delta}_0$ and curves in $|q^*\Gamma|$ tangent to $D'$) when $d=\frac{n+1}{3}$. 
        \item Our result does not apply to the stratum where $d=\frac{n+3}{3}$.
    \end{itemize}
    The above follows from Corollary~\ref{cor:hirzebruch-2pt-2comp-horikawa-cover}.
    In the smaller strata where $d>\frac{n+4}{3}$, a very general surface contains $n+3-2d$ rational curves (from $\widetilde{\Delta}_0$ plus $n+2-2d$ singular points on $D$, $\rho^{-1}(\widetilde{\Gamma})$ is contracted by $\pi$) and $(8n+4d+24)+2+(n+2-2d)=9n+2d+28$ elliptic curves ($E_x, E_y$, and curves in $|q^*\Gamma|$ tangent to $D$ or intersects $D$ at a node). 
    This follows from Corollary~\ref{cor:hirzebruch-2pt-3comp-horikawa-cover}.
    \item For the second component (type $\left(\frac{n+2}{2}\right)$), a very general surface contains 1 rational curve (from $\widetilde{\Delta}_0$, $\rho^{-1}(\widetilde{\Gamma})$ is contracted by $\pi$) and $(8n+4d+24)+2=10n+30$ elliptic curves ($E_x, E_y$, and curves in $|q^*\Gamma|$ tangent to $D'$) by Corollary~\ref{cor:hirzebruch-2pt-3comp-horikawa-cover}.
\end{itemize}

\item \textbf{When $4\nmid n$ and $n\ge 6$:} The moduli space consists of a single component, and the numbers are the same as the first component in the previous case. See Corollaries~\ref{cor:hirzebruch-2pt-2comp-horikawa-cover} and \ref{cor:hirzebruch-f0-2pt-2comp-horikawa-cover}.
\end{itemize}

This is summarized in Table~\ref{table:horikawa-count-2}.
In all cases described above, the geometrically elliptic curves arise either as $\rho^{-1}(\widetilde{\Delta}_0)$, $\rho^{-1}(E_x)$, $\rho^{-1}(E_y)$, or as pullbacks of rational curves in the fiber class $q^*\Gamma$ 
that are either simply tangent to the branch locus $D$ or intersects it at a node (and additionally in class $\widetilde{\Gamma}$ for type $(1^*)$). 
The rational curves arise strictly as $\rho^{-1}(\widetilde{\Delta}_0)$, or as the exceptional divisors of blown-up nodes on the branch locus $D$.

\begin{table}[htbp]
\centering
\renewcommand{\arraystretch}{1.3}
\begin{tabular}{clccl}
\toprule
\textbf{Component} & \textbf{Stratum} & \textbf{Condition} & \textbf{Rational} & \textbf{Elliptic} \\
\midrule
$n=4$ & Type $(1)$ & -- & $0$ & $70$ \\
\midrule
\multirow{2}{*}{$n=4$} 
& Type $(1^*)$ & -- & $0$ & $2$ \\
& Type $(3')$ & -- & $?$ & $?$ \\
\midrule
\multirow{2}{*}{$n=5$} 
& Type $(0)$  & -- & $0$ & $80$  \\
& Type $(2)$ & -- & $0$ & $81$  \\
\midrule
$n=5$ & Type $(2^*)$ & -- & $0$ & $0$ \\
\midrule
\multirow{2}{*}{\parbox{2cm}{\centering $n \ge 8$\\ $4 \mid n$}} 
& Type $(d)$  & $d< \frac{n+1}{3}$ & $0$ & $10n+30$  \\
& Type $(d)$  & $d= \frac{n+1}{3}$ & $0$ & $10n+31$  \\
& Type $(d)$  & $d= \frac{n+3}{3}$ & $?$ & $?$  \\
& Type $(d)$ & $\frac{n+4}{3}<d\le \frac{n-2}{2}$ & $n+3-2d$ & $9n+2d+28$ \\
\midrule
\parbox{2cm}{\centering $n \ge 8$\\ $4 \mid n$}& Type $\left(\frac{n+2}{2}\right)$ & -- & $1$ & $10n+30$  \\
\midrule
\multirow{2}{*}{\parbox{2cm}{\centering $n \ge 6$\\ $4 \nmid n$}} 
& Type $(d)$  & $d< \frac{n+1}{3}$ & $0$ & $10n+30$  \\
& Type $(d)$  & $d= \frac{n+1}{3}$ & $0$ & $10n+31$  \\
& Type $(d)$  & $d= \frac{n+3}{3}$ & $?$ & $?$  \\
& Type $(d)$ & $\frac{n+4}{3}<d$ & $n+3-2d$ & $9n+2d+28$ \\
\bottomrule
\end{tabular}
\vspace{.2in}
\caption{Number of rational and elliptic curves on very general Horikawa surfaces of the second kind and $n\ge4$ in each strata and component of moduli.
Strata are listed in decreasing order of dimension.}
\label{table:horikawa-count-2}
\end{table}

\bibliographystyle{alpha}
\bibliography{references}
\end{document}